\documentclass[reqno,9pt]{amsart}
\title{Weighted words at degree two, II: flat partitions, regular partitions, and application to level one perfect crystals}
\author{Isaac KONAN} \address{IRIF \\ University
of Paris\\  Paris,  75013, France}
\email{konan@irif.fr}
\date{}

\usepackage{amssymb}
\usepackage{amsmath}
\usepackage{amsthm}
\usepackage{todonotes}
\usepackage{fullpage}
\usepackage{enumitem}
\usepackage{xcolor}
\usepackage{graphicx}
\usepackage{pgfarrows,pgfnodes}
\usepackage{url}
\usepackage{textcomp}
\usepackage{ytableau}
\usepackage{tikz}
\usepackage[toc,page]{appendix}

\allowdisplaybreaks
\newcommand{\m}{\medbreak}
\newcommand{\bi}{\bigbreak}

\definecolor{foge}{rgb}{0.1, 0.6, 0.1}

\numberwithin{equation}{section}
\numberwithin{figure}{section}
\newtheorem{fig}[figure]{Figure}
\newcommand{\Thm}[1]{Theorem \ref{#1}}
\newcommand{\Cor}[1]{Corollary \ref{#1}}
\newcommand{\Lem}[1]{Lemma \ref{#1}}
\newcommand{\Prp}[1]{Proposition \ref{#1}}
\newcommand{\Sct}[1]{Section \ref{#1}}
\newcommand{\Def}[1]{Definition \ref{#1}}

\newcommand{\Expl}[1]{Example \ref{#1}}

\newcommand{\Ll}{\Lambda}
\newcommand{\Pp}{\mathcal{P}}
\newcommand{\Ppp}{\Pp^{\gg}_{\co}}

\newcommand{\Reg}{\mathcal{R}^{\ep,\co}_{1}}
\newcommand{\Regg}{\mathcal{R}^{\ep,\co}_{2}}
\newcommand{\Regk}{\mathcal{R}^{\ep,\co}_{k}}
\newcommand{\Flt}{\mathcal{F}^{\ep,\co}_{1}}
\newcommand{\Fltt}{\mathcal{F}^{\ep,\co}_{2}}
\newcommand{\Fltk}{\mathcal{F}^{\ep,\co}_{k}}
\newcommand{\Z}{\mathbb{Z}}

\newcommand{\C}{\mathcal{C}}
\newcommand{\Cc}{\mathcal{C}_\ltimes}
\newcommand{\B}{\mathcal{B}}

\newcommand{\Rr}{\mathcal{R}}
\newcommand{\ep}{\epsilon}
\newcommand{\gte}{\gtrdot_{\ep}}
\newcommand{\gtei}{\gtrdot_{\ep_2}}

\newcommand{\wt}{\overline{\mathrm{wt}}}
\newcommand{\Odd}{\mathcal{O}_{\ep}}
\newcommand{\Sc}{\mathcal{S}_{\ep}}

\newcommand{\od}{\succ_{\ep}}
\newcommand{\odp}{\gg_{\ep}}
\newcommand{\odg}{\gg^{\ep}}

\newcommand{\ot}{\otimes}
\newcommand{\p}{\mathfrak p}

\newcommand{\I}{\{0,\ldots,n\}}
\newcommand{\E}{\mathcal{E}_{\ep}}

\newcommand{\la}{\lambda}
\newcommand{\sss}{\{0,\ldots,s\}}
\newcommand{\ssss}{\{0,\ldots,s-1\}}

\newcommand{\co}{c_{g}}

\numberwithin{equation}{section}
\newtheorem{theo}{Theorem}[section]
\newtheorem{prop}[theo]{Proposition}
\newtheorem{lem}[theo]{Lemma}
\newtheorem{cor}[theo]{Corollary}
\newtheorem{rem}[theo]{Remark}
\newtheorem{ex}[theo]{Example}

\theoremstyle{definition} \newtheorem{deff}[theo]{Definition}

\begin{document}
\maketitle
\begin{abstract}
In a recent work, Keith and Xiong gave a refinement of Glaisher's theorem by using a Sylvester-style bijection. In this paper, we introduce two families of colored partitions, flat and regular partitions, and generalize the bijection of Keith and Xiong to these partitions. We then state two results, the first at degree one, where partitions have parts with primary colors, and the second result at degree two for secondary-colored partitions, using the result of the first paper of this series on Siladi\'c's identity. 
These results allow us to easily retrieve the Frenkel-Kac character formulas of level one standard modules for the type $A_{2n}^{(2)}, D_{n+1}^{(2)}$ and $B_n^{(1)}$, and also to make the connection between the result stated in paper one and the representation theory.
\end{abstract}
\section{Introduction}
\subsection{History}
Let $n$ be a positive integer. A partition of $n$ is defined as a non-increasing sequence of positive integers, called the parts of the partition, and whose sum is equal to $n$. For example, the partitions of $5$ are 
\[(5),(4,1),(3,2),(3,1,1),(2,2,1),(2,1,1,1), \,\,\text{and}\,\,(1,1,1,1,1,1,1)\,\cdot\]
A partition identity is a combinatorial identity that links two or several sets of integer partitions. The study of such identities has interested mathematicians for centuries, dating back to
Euler's proof that there are as many partitions of $n$ into distinct parts as partitions of $n$ into odd parts. The Euler distinct-odd identity can be written in terms of $q$-series with the following expression:
\begin{equation}\label{eq:euler}
(-q;q)_\infty = \frac{(q^2;q^2)_\infty}{(q;q)_\infty}=\frac{1}{(q;q^2)_\infty}\,\cdot
\end{equation}
In the latter formula, $(x;q)_m = \prod_{k=0}^{m-1} (1-xq^k)$
for  any $m\in \mathbb{N}\cup \{\infty\}$ and $x,q$ such that $|q|<1$. We also define for any $a_1,\ldots,a_s$ the expression $(a_1,\ldots,a_s;q)_m = (a_1;q)_m \cdots (a_s;q)_m$. 
\m We note that while the Euler identity is not difficult to prove by computing the generating function of both sets of partitions, the bijection that links these sets is not trivial.
\m
In \cite{G83}, Glaisher bijectively proved the first broad generalization of the Euler identity. Let $m$ be a positive integer. We define an $m$-flat partition to be a partition where
the differences between two consecutive parts, as well the smallest part, are less than $m$, and an $m$-regular partition to be a partition with parts not divisible by $m$.
The generalization of Euler's identity given by Glaisher, and  which makes the connection between $m$-flat and $m$-regular partitions, is stated in the following theorem.
\begin{theo}[Glaisher]
For a fixed positive integer $n$, the following sets of partitions are equinumerous:
\begin{enumerate}
\item the $m$-regular partitions of $n$,
\item the partitions of  $n$ with fewer than $m$ occurrences for each positive integer,
\item the $m$-flat partitions of $n$. 
\end{enumerate}
The corresponding $q$-series is 
\begin{equation}
\prod_{n\geq 1} (1+q^n+ q^{2n}+\cdots+q^{n(m-1)}) = \prod_{\substack{n\geq 1 \\ m \nmid n}} \frac{1}{(1-q^n)} = \frac{(q^m;q^m)_{\infty}}{(q;q)_{\infty}} \,\cdot
\end{equation}
\end{theo}
Another bijective proof of the Euler identity was given by Sylvester \cite{Sy82}. However, it was a open problem to find a suitable generalization of Sylvester's bijection for the Glaisher identity. 
\bi
This problem was solved, a century after the paper of Sylvester, by Stockhofe in his Ph.D thesis \cite{Sto82}.  
In the 90's, seminal works of Bessenrodt \cite{Bes94}, and Pak and Postinkov \cite{PaPo98}, related the Sylvester algorithm to the alternating sign sum of integer partitions. They then gave new refinements of the Euler identity.
\bi In this paper, we focus on a broad refinement of Glaisher's identity given by Keith and Xiong \cite{KX19}. Their proof used a variant of the Sylvester-style bijection given by Stockhofe.
\begin{theo}[Keith-Xiong]\label{theo:kx}
Let $m\geq 2$, $u_1,\ldots,u_{m-1}$ be non-negative integers, and let $n$ be a positive integer. Then, the number of $m$-flat partitions of $n$  \textbf{with} $u_i$ parts congruent to $i\mod m$ is equal to the number of $m$-regular partitions of $n$ \textbf{into} $u_i$ parts congruent to $i\mod m$.  
\end{theo}
\bi The main goal of this second paper is to state two results, beyond the refinement of Keith-Xiong, that link a general definition of flat partitions and regular partitions given in terms of weighted words.
\subsection{Statement of Results}
\subsubsection{Generalized flat and regular partitions}
We first recall the notion of grounded partitions, introduced by the author and Dousse in \cite{DK2}.
\bi 
Let $\C$ be a set of colors, and let $\Z_{\C} = \{k_c : k \in \Z, c \in \C\}$ be the set of colored parts. We here refer to $|k_c|=k$ and $c(k_c)=c$ respectively as the size and the color of the colored part $k_c$.
\begin{deff}
Let $\gg$ be a binary relation defined on $\Z_{\C}$.
A \textit{generalized colored partition} with relation $\gg$ is a finite sequence $(\pi_0,\ldots,\pi_s)$ of parts, where for all $i \in \{0,\ldots,s-1\},$ $\pi_i\gg\pi_{i+1}.$
\end{deff}
In the following, the quantity $|\pi|=|\pi_0|+\cdots+|\pi_s|$ is the size of the partition $\pi$, and $C(\pi) = c(\pi_0)\cdots c(\pi_s)$ is its color sequence. 
Let us choose a particular color $\co$. We now introduce the grounded partitions as a special kind of generalized colored partitions.
\begin{deff}
A \textit{grounded partition} with ground $\co$ and relation $\gg$ is a non-empty generalized colored partition $\pi=(\pi_0,\ldots,\pi_s)$ with relation $\gg$, such that $\pi_s = 0_{\co}$, and when $s>0$, $\pi_{s-1}\neq 0_{\co}.$
\noindent Let $\Ppp$ denote the set of such partitions.
\end{deff}
In the following, we explicitly write $\pi=(\pi_0,\ldots,\pi_{s-1},0_{\co})$. The trivial partition in $\Ppp$ is then $(0_{\co})$.  
\begin{ex}Consider the set of classical integer partitions $\pi=(\pi_1,\ldots,\pi_s)$, with integer parts satisfying $\pi_1\geq \cdots \geq \pi_s>0$, where the empty partition $\emptyset$ is such that $s=0$.
This set is in bijection 
with the set $\Pp_{c}$ of the grounded colored partitions with only one color $c$, defined by the relation 
\begin{equation}
 k_c\gg l_c \Longleftrightarrow k-l\geq 0\,\cdot
\end{equation}
In fact, the correspondence is defined by 
\begin{equation*}
 (\pi_1,\ldots,\pi_s) \mapsto ((\pi_1)_c,\ldots,(\pi_s)_c,0_c)=(\pi'_0,\cdots,\pi'_{s-1},0_c)\,,
\end{equation*}
where the empty partition $\emptyset$ corresponds to the colored part $0_c$.
\end{ex}
\bi Note that from a set $\C$, one can build the \textit{colors of degree} $k$ as products of $k$ colors in $\C$, for any positive integer $k$. The set of degree $k$ colors is denoted $\C^k$, and is equal to $\{c_1\cdots c_k: c_1,\ldots, c_k \in \C\}$. Therefore, $\C^1=\C$ refers to the set of \textit{primary} colors, and we say that we use weighted words at degree $k$ if the colors of the parts have degree at most $k$, i.e if the set of colors equals $\C^1 \sqcup \cdots \sqcup \C^k$. 
\bi 
We now extend the notion of flatness to the grounded partitions.
\begin{deff}
Let $\ep$ be a function from $\C^2$ to $\Z$, called \textit{energy}. A \textit{flat partition} with ground $\co$ and energy $\ep$ is a grounded partition with ground $\co$ and the relation  $\gte$ defined by
\begin{equation}
k_c \gte l_d \Longleftrightarrow k-l = \ep(c,d)\,\cdot
\end{equation}
\end{deff}
These partitions are determined by their sequence of colors as well as the energy $\ep$. This comes from the fact that for such a partition $\pi=(\pi_0,\ldots,\pi_{s-1},0_{\co})$, the  computation of the size of $\pi_k$ gives the following relation,
\[|\pi_k|=\sum_{l=k}^{s-1}\ep(c(\pi_l),c(\pi_{l+1}))\,\cdot\]
\begin{ex}\label{ex:mflat}A good example of flat partitions are the $m$-flat partitions. It suffices to consider the set of colors 
$\C =\{c_0,\ldots, c_{m-1}\}$, $\co=c_0$ and define the energy $\ep$ by 
\[\ep(c_i,c_j) = 
\begin{cases}
i-j \ \ \ \ \ \ \ \ \text{if}\ \ i\geq j\\
m+i-j\ \ \text{if}\ \ i< j
\end{cases}\,\cdot\]
With these definitions, for any flat partition, its parts with color $c_i$ necessarily have a size congruent to  $i$ modulo $m$. Moreover, observe that $\ep$ has values in $\{0,\ldots,m-1\}$. We then associate to any $m$-flat partition $\la = (\la_1,\ldots,\la_s)$ the flat partition $\pi=(\pi_0,\ldots,\pi_{s-1},0_{c_0})$ such that, for all $k\in \{0,\ldots,s-1\}$, 
\[|\pi_{k}|=\la_{k+1}\quad\text{and}\quad c(\pi_k)= c_{[\la_{k+1}]_m}\,,\]
where $[\la_{k+1}]_m = \la_{k+1} \mod m$.
\end{ex}
\bi 
Let us now generalize the notion of regularity. 
\begin{deff}
A $c$-regular partition with ground $\co$ and with relation $\gg$ is a grounded partition $\pi=(\pi_0,\ldots,\pi_{s-1},0_{\co})$ with ground $\co$ and relation $\gg$, such that $c(\pi_k)\neq c$ for all $k\in \{0,\ldots,s-1\}$. When $c=\co$, it is called a regular partition in $\co$.
\end{deff}
\begin{ex}\label{ex:mreg}An example of such partitions are the $m$-regular partitions.  It suffices to consider as in \Expl{ex:mflat} the set of colors 
$\C =\{c_0,\ldots, c_{m-1}\}$, $\co=c_0$ and define the relation $\gg$ by 
\[k_{c_i}\gg l_{c_j} \Longleftrightarrow k\geq l \quad\text{and}\quad k-l\equiv i-j \mod m\,,\]
so that, in any regular partition, the size of parts with color $c_i$ are necessarily congruent to $i$ modulo $m$.
We then associate to any $m$-regular partition $\la = (\la_1,\ldots,\la_s)$ the regular partition $\pi=(\pi_0,\ldots,\pi_{s-1},0_{c_0})$ in $c_0$ and such that, for all $k\in \{0,\ldots,s-1\}$, 
\[\pi_{k}=\la_{k+1}\quad\text{and}\quad c(\pi_k)= c_{[\la_{k+1}]_m}\,\cdot\]
\end{ex}
\bi
In the remainder of this paper, we take for regular partitions a relation $\odp$, associated to a energy $\ep$, and defined by
\begin{equation}
k_c \odp l_d \Longleftrightarrow k-l \geq \ep(c,d)\,\cdot
\end{equation}
We then  call the relation $\odp$ the \textit{minimal difference condition} given by energy $\ep$, and a regular partition with relation $\odp$ is referred to as regular partition with energy $\ep$. 
\m
Considering the set of colors and the energy $\ep$ given \Expl{ex:mflat}, by \Thm{theo:kx}, there exists a bijection between the corresponding flat partitions with ground $c_0$ and energy $\ep$ and the regular partitions in $c_0$ and with energy $\ep$, such that the parts with color $c_i$ have sizes congruent to $i \mod m$. The latter regular partitions are those described in \Expl{ex:mreg}. Furthermore, the bijection occurs between the partitions of both kinds with a fixed size and numbers of occurrences of the colors different from the ground $c_0$.
In this spirit, the main theorems of this paper make a connection between the flatness and the regularity, and have the following formulation.
\begin{theo}\label{theo:flatreg}
Let $\C$ be a set of colors and let $\co\in \C$ be the ground. Then,
for some suitable energies $\ep'$ and $\ep$, there exists a bijection between \textbf{a certain} set of flat partitions with ground $\co$ and energy $\ep$ and \textbf{a certain} set of regular partitions in ground $\co$ and with energy $\ep'$.
\end{theo}
\bi
Two such theorems are stated in Section 2. The first result, \Thm{theo:flatreg1}, is related to weighted words at degree one, i.e. for primary colors and parts, and energies satisfying $\ep=\ep'$. We here give a corollary of \Thm{theo:flatreg1} as the following analogue of Glaisher's theorem for $m$-regular partitions into distinct parts.
\bi 
\begin{cor}
\label{cor:analogueglaisher}
Let $m$ and $n$ be positive integers. Then, the number of $m$-regular partitions of $n$ into \textbf{distinct} parts is equal to the number of $(m+1)$-flat partitions of $n$, such that
\begin{itemize}
\item the smallest part is less than $m$,
\item two consecutive parts divisible by $m$ are necessarily equal,
\item two consecutive parts not divisible by $m$ and with the same congruence modulo $m$ are necessarily distinct.
\end{itemize}
\end{cor}
\begin{ex}Here we take $m=3$ and $n=16$, and the $3$-regular partition of $16$ into distinct parts are
\[(16),(14,2),(13,2,1),(11,5),(11,4,1),(10,5,1),(10,4,2),(8,7,1),(8,5,2,1),\text{  and  }(7,5,4)\,\]
and the $4$-flat partitions of $16$ of the second kind are
\[(8,5,2,1),(7,5,3,1),(7,4,3,2),(6,5,4,1),(6,5,3,2),(6,4,3,2,1),(5,4,3,3,1),(5,3,3,3,2),(4,3,3,3,2,1),\]
\[\text{  and  }(3,3,3,3,3,1)\,\cdot\]
\end{ex}
Another consequence of \Thm{theo:flatreg1} consists in easing the computation of characters of the representations of some affine Lie algebras. Formulas for such characters are given in \cite{DK2} as generating functions of the flat partitions with respect to a ground and a suitable energy function of perfect crystals.
\bi 
The second result, \Thm{theo:flatreg2}, concerns weighted words at degree two, and energies satisfying $\ep=\ep'$ up to some exceptions. 
This second theorem uses \Thm{theo:flatreg1} and the result from the first paper of this series \cite{IK201}, which we summarize in \Sct{sect:paper1}. 
In the particular case of representations of affine Lie algebras we study here, \Thm{theo:flatreg2} allows us to connect the difference conditions
of the result on Siladi\'c's identity, given in the first paper, and the energy function of the square, in terms of tensor product, of the vector representation of $A_{2n}^{(2)}$. We here state a corollary of \Thm{theo:flatreg2}, presented as a companion of Siladi\'c's identity.
\begin{cor}\label{cor:silcomp}
Let $n$ be a non-negative integer. Denote by $A(n)$ the number of partitions $\la=(\la_1,\ldots,\la_s)$ of $n$, into parts different from $2$, such that 
$\lambda_i -\lambda_{i+1}\geq 5$ and 
\begin{align*}
\lambda_i -\lambda_{i+1}  = 5 &\Rightarrow \, \lambda_i+\lambda_{i+1} \equiv \pm 3 \mod 16\,,\\
\lambda_i -\lambda_{i+1}  = 6 &\Rightarrow \, \lambda_i+\lambda_{i+1}  \equiv 0,\pm 4, 8\mod 16\,,\\
\lambda_i -\lambda_{i+1}  = 7 &\Rightarrow \, \lambda_i+\lambda_{i+1}  \equiv \pm 1, \pm 5, \pm 7 \mod 16\,,\\
\lambda_i -\lambda_{i+1}  = 8 &\Rightarrow \, \lambda_i +\lambda_{i+1} \equiv 0,\pm 2,\pm 6, 8\mod 16\,\,\cdot
\end{align*}
Denote by $B(n)$ the number of partitions $\la=(\la_1,\ldots,\la_s)$ of $n$, such that the parts congruent to $0 \mod 8$ can be overlined, the last part is less than $11$ and different from $2$ and $\overline{8}$, and $0\leq \lambda_i -\lambda_{i+1}\leq 16$ with the additional conditions
\begin{align*}
\lambda_i -\lambda_{i+1}  = 0&\Rightarrow \, \lambda_i+\lambda_{i+1} \equiv \overline{\overline{0}}\mod 16\,,\\
\lambda_i -\lambda_{i+1}  = 1 &\Rightarrow \, \lambda_i+\lambda_{i+1} \equiv \pm \overline{1} \mod 16\,,\\
\lambda_i -\lambda_{i+1}  = 2 &\Rightarrow \, \lambda_i+\lambda_{i+1} \equiv 0 \mod 16\,,\\
\lambda_i -\lambda_{i+1}  = 3 &\Rightarrow \, \lambda_i+\lambda_{i+1} \equiv \pm \overline{3} \mod 16\,,\\
\lambda_i -\lambda_{i+1}  = 4 &\Rightarrow \, \lambda_i+\lambda_{i+1} \equiv \pm 2, \pm \overline{4} \mod 16\,,\\
\lambda_i -\lambda_{i+1}  = 5 &\Rightarrow \, \lambda_i+\lambda_{i+1} \equiv \pm 3, \pm \overline{5}\mod 16\,,\\
\lambda_i -\lambda_{i+1}  = 6 &\Rightarrow \, \lambda_i+\lambda_{i+1}  \equiv 0,\pm 4, 8, \pm \overline{6}\mod 16\,,\\
\lambda_i -\lambda_{i+1}  = 7 &\Rightarrow \, \lambda_i+\lambda_{i+1}  \equiv \pm 1, \pm 5, \pm 7 , \pm \overline{7} \mod 16\,,\\
\lambda_i -\lambda_{i+1}  = 8 &\Rightarrow \, \lambda_i +\lambda_{i+1} \equiv 0,\pm 2,\pm 6, 8 , \overline{8}\mod 16\,\,,
\end{align*}
where the number lines of $\lambda_i+\lambda_{i+1}\mod 16$ indicates the numbers of overlined parts $\la_i,\la_{i+1}$. 
We then have that $A(n)=B(n)$, and the corresponding identity his
\begin{equation}
\sum_{n\geq 0} B(n)q^n = \sum_{n\geq 0} A(n)q^n = (-q;q^2)_{\infty}\,\cdot
\end{equation}
\end{cor}
\subsubsection{Character formulas for level one standard modules}
We refer the reader to \cite{HK02} for the definitions in the theory of Kac-Moody algebras.
\m
Let $n$ be a non-negative integer, and consider the Cartan datum $(A,\Pi,\Pi^\vee,P,P^\vee)$ for a generalised Cartan matrix $A$ of affine type and rank $n$.
Here $\Pi$ is the set of the simple roots $\alpha_i (i \in \{0, \dots , n\})$, and 
denote by $\bar P = \Z\Lambda_0 \oplus \cdots \oplus \Z\Lambda_{n}$ the lattice of the classical weights, where the elements $\Lambda_{\ell}$ $(\ell \in \{0, \dots , n\})$ are the 
fundamental weights. Denote by $\delta$ the null root. $L(\Lambda)$ denote the irreducible module of highest weight $\Lambda$, also called the standard module. Using \Thm{theo:flatreg1} and \cite[Theorem 3.8]{DK2}, we retrieve the Frenkel-Kac character formulas \cite{FK80} for the following. 
\begin{theo}\label{theo:char1}
Let $n\geq 2$, and let $\Lambda_0, \dots, \Lambda_{n}$ be the fundamental weights  and let $\alpha_0, \dots, \alpha_{n}$ be the simple roots of $A_{2n}^{(2)}$.
We have in $\Z[[e^{-\alpha_0},e^{-\alpha_1},\cdots,e^{-\alpha_{n}}]]$ that
\begin{equation}
e^{-\Ll_0}\mathrm{ch}(L(\Ll_0)) = \prod_{u=1}^n (-e^{-\delta'-\frac{1}{2}\alpha_n-\sum_{i=u}^{n-1}\alpha_i},-e^{-\delta'+\frac{1}{2}\alpha_n+\sum_{i=u}^{n-1}\alpha_i};e^{-2\delta'})_\infty\,,
\end{equation}
where $2\delta'= \delta= 2\alpha_0 + \cdots + 2\alpha_{n-1}+\alpha_n$ is the null root.
\end{theo}
\begin{theo}\label{theo:char2}
Let $n\geq 2$, and let $\Lambda_0, \dots, \Lambda_{n}$ be the fundamental weights  and let $\alpha_0, \dots, \alpha_{n}$ be the simple roots of $D_{n+1}^{(2)}$.
We have in $\Z[[e^{-\alpha_0},e^{-\alpha_1},\cdots,e^{-\alpha_{n}}]]$ that
\begin{align}
e^{-\Ll_0}\mathrm{ch}(L(\Ll_0)) &= \frac{1}{(e^{-\delta};e^{-2\delta})_{\infty}}\cdot\prod_{u=1}^n (-e^{-\delta-\sum_{i=u}^{n}\alpha_i},-e^{-\delta+\sum_{i=u}^{n}\alpha_i};e^{-2\delta})_\infty\,,
\\e^{-\Ll_n}\mathrm{ch}(L(\Ll_n)) &= \frac{1}{(e^{-\delta};e^{-2\delta})_{\infty}}\cdot\prod_{u=1}^n (-e^{-\sum_{i=u}^{n}\alpha_i},-e^{-2\delta+\sum_{i=u}^{n}\alpha_i};e^{-2\delta})_\infty\,,
\end{align}
where $\delta= \alpha_0+\cdots+\alpha_n$ is the null root.
\end{theo}
\begin{theo}\label{theo:char3}
Let $n\geq 3$, and let $\Lambda_0, \dots, \Lambda_{n}$ be the fundamental weights  and let $\alpha_0, \dots, \alpha_{n}$ be the simple roots of $B_{n}^{(1)}$.
We have in $\Z[[e^{-\alpha_0},e^{-\alpha_1},\cdots,e^{-\alpha_{n}}]]$ that
\begin{equation}
e^{-\Ll_n}\mathrm{ch}(L(\Ll_n)) = \frac{1}{(e^{-\delta};e^{-2\delta})_{\infty}}\cdot\prod_{u=1}^n (-e^{-\sum_{i=u}^{n}\alpha_i},-e^{-\delta+\sum_{i=u}^{n}\alpha_i};e^{-\delta})_\infty\,,
\end{equation}
where $\delta= \alpha_0+\alpha_1+2\alpha_2\cdots+2\alpha_n$ is the null root.
\end{theo}
These product formulas of the characters can also be found in the Wakimoto's book \cite{Wak01}
\bi
The remainder of the paper is organized as follows. We first provide in \Sct{sect:setup} the tools and the main result of the first paper \cite{IK201}, as well as the main results connecting flat and regular partitions at degree one and two. Second, assuming  \Thm{theo:flatreg1}, we recover in \Sct{sect:level1} the Frenkel-Kac character formulas.
After that, we prove \Thm{theo:flatreg1} and \Thm{theo:flatreg2} respectively in \Sct{sect:deg1} and \Sct{sect:deg2}.
We finally discuss in \Sct{sect:degk} the possibility of a suitable \Thm{theo:flatreg} at degree $k$ for $k\geq 3$ and conclude with some remarks in \Sct{sect:remarks}.
\section{The setup}\label{sect:setup}
\subsection{Weighted words: parts as energetic particles
}\label{sect:paper1}
We here recall the basic tools and results of the first paper \cite{IK201}.
\m Let $\C$ be a set of primary colors, countable or not. Recall the set of primary parts $\Z_\C$, which is also denoted by $\Pp=\Z \times \C$. A primary part with size $k$ and color $c$ is then identified as $k_c$ or $(k,c)$.
\begin{deff}\label{def:minerg}
A \textit{minimal energy} is an energy $\ep$  from $\C^2$ to $\{0,1\}$. When $\C=\{c_1,\ldots,c_n\}$, the data given by $\ep$ is equivalent to the matrix $M_\ep  = (\ep(c_i,c_j))_{i,j=1}^n$, denoted the \textit{energy matrix}.
The \textit{energy} relation $\od$ with respect to $\ep$ is the binary relation on $\Pp^2$ defined by the following, 
\begin{equation}\label{rel}
 (k,c)\od (k',c') \Longleftrightarrow \, k-k'\geq \ep(c,c')\,\cdot
\end{equation}
\end{deff}
\begin{ex}\label{ex:ordered1}
Let $\C=\{c_1,\ldots,c_n\}$ be a set of colors, and let $I_1,I_2$ be a set-partition of $\{1,\ldots,n\}$. We then define the minimal energy
\begin{equation}
\ep(c_i,c_j) = 
\begin{cases}
\chi(i<j) \ \ \text{if}\ \ i\neq j \\
\chi(i\in I_1)\ \ \text{if}\ \ i=j
\end{cases}\,\cdot
\end{equation}
Then, the general order on the parts is
$$\cdots\od (k+1)_{c_1}\od k_{c_n}\od k_{c_{n-1}}\od\cdots \od k_{c_2}\od k_{c_1}\od \cdots\,,$$
and $k_{c_i}\od k_{c_i}$ if $i\in I_2$ and $k_{c_i}\not\od k_{c_i}$ if $i \in I_1$. This means that a part  $k_{c_i}$ can be repeated  in the well-ordered sequence of parts if and only if $i\in I_2$.  
\end{ex}
\begin{ex}\label{ex:ordered2}
Suppose that $\C=\{c_1,c_2\}$ and define the minimal energy
\begin{equation}
\ep(c_i,c_j) = \chi(i=j)\,\cdot
\end{equation}
We then have the general relation on the parts,
\[\cdots\od (k+1)_{c_2}\od k_{c_2}\od k_{c_{1}}\od k_{c_2} \od k_{c_1} \od (k-1)_{c_1}\od \cdots\,,\]
and $k_{c_i}\not\od k_{c_i}$. A well-related sequence of parts with the same size is then an  alternating sequence.
\end{ex}
\begin{deff}\label{def:secpar}
We define the \textit{secondary parts} to be the sums of two consecutive primary parts in terms of $\od$. Denote by $\Sc = \Z\times \C^2$ the set of secondary parts, in such a way that the part
\begin{equation}
 (k,c,c') = (k+\ep(c,c'),c)+(k,c') 
\end{equation}
has size $2k+\ep(c,c')$ and color $cc'$. In the following, we identify  a secondary part as $(k,c,c')$ or $(2k+\ep(c,c'))_{cc'}$. We denote by $\gamma(k,c,c')$ and $\mu(k,c,c')$ the primary parts
\[\gamma(k,c,c')=(k+\ep(c,c'),c)\quad \text{and}\quad \mu(k,c,c') = (k,c')\,,\]
respectively called \textit{upper} and \textit{lower} halves of the secondary parts $(k,c,c')$.
\end{deff}
\begin{deff}\label{def:relation2}
We define the relation $\odp$ on $\Pp\sqcup \Sc$ as follows: 
 \begin{align}
(k,\textcolor{red}{c})\odp (k',\textcolor{foge}{c'}) &\Longleftrightarrow k-k' > \ep(\textcolor{red}{c},\textcolor{foge}{c'})\,,\label{pp}\\
(k,\textcolor{red}{c})\odp(k',\textcolor{foge}{c'},\textcolor{blue}{c''}) &\Longleftrightarrow k - (2k'+\ep(\textcolor{foge}{c'},\textcolor{blue}{c''}))\geq \ep(\textcolor{red}{c},\textcolor{foge}{c'})+\ep(\textcolor{foge}{c'},\textcolor{blue}{c''})\,,\label{ps}\\
(k,\textcolor{red}{c},\textcolor{foge}{c'})\odp (k',\textcolor{blue}{c''})&\Longleftrightarrow (2k+\ep(\textcolor{red}{c},\textcolor{foge}{c'}))-k' > \ep(\textcolor{red}{c},\textcolor{foge}{c'})+\ep(\textcolor{foge}{c'},\textcolor{blue}{c''})\,,\label{sp}\\
(k,\textcolor{red}{c},\textcolor{foge}{c'})\odp (k',\textcolor{blue}{c''},\textcolor{purple}{c'''}) &\Longleftrightarrow k-k' \geq \ep(\textcolor{foge}{c'},\textcolor{blue}{c''})+\ep(\textcolor{blue}{c''},\textcolor{purple}{c'''})\,\cdot\label{ss}
\end{align}
\end{deff} 
\begin{deff}\label{def:rho}
Let $\Odd$ (respectively $\E$) be the set of all generalized colored partitions with parts in $\Pp$ (resp. $\Pp\sqcup\Sc$)
and with relation $\od$ (resp. $\odp$). 
\bi
For $\rho\in\{0,1\}$, we consider the following sets:
\begin{itemize}
 \item $\Pp^{\rho_+} =\Z_{\geq \rho}\times\C \text{  and  } \Sc^{\rho_+} = \Z_{\geq \rho}\times\C^2 =\{(k,c,c')\in \Sc:\,k\geq \rho\}$,
 \item $\Pp^{\rho_-} =\Z_{\leq \rho}\times\C \text{  and  } \Sc^{\rho_-} =\{(k,c,c')\in \Sc:\,k+\ep(c,c')\leq \rho\}$.
\end{itemize}
We then denote by $\Odd^{\rho_+}$ (respectively $\Odd^{\rho_-}$) the subset of $\Odd$ with parts in $\Pp^{\rho_+}$ (respectively $\Pp^{\rho_-}$), and by 
$\E^{\rho_+}$ (respectively $\E^{\rho_-}$) the subset of $\E$ with parts in $\Pp^{\rho_+}\sqcup \Sc^{\rho_+}$ (respectively $\Pp^{\rho_-}\sqcup \Sc^{\rho_-}$).
\end{deff}
\bi 
Since secondary colors are products of two primary colors, the color sequence of partitions in $\Odd$ and $\E$ is then seen as a finite \textit{non-commutative} product of colors in $\C$.
\bi 
The main theorem of \cite{IK201} is then stated as follows.
\begin{theo}\label{theo:degree2}
 For any integer $n$ and any finite non-commutative product  $C$ of colors in $\C$,
 there exists a bijection between $\{\la \in \Odd: (C(\la),|\la|)=(C,n)\}$ and $\{\nu \in \E: (C(\nu),|\nu|)=(C,n)\}$. 
 In particular, for $\rho \in \{0,1\}$, we have the identities
 \begin{align}
 |\{\nu \in \E^{\rho_+}: (C(\nu),|\nu|)=(C,n)\}| &= |\{\la \in \Odd^{\rho_+}: (C(\la),|\la|)=(C,n)\}|\,,\\
 |\{\nu \in \E^{\rho_-}: (C(\nu),|\nu|)=(C,n)\}| &= |\{\la \in \Odd^{\rho_-}: (C(\la),|\la|)=(C,n)\}|\,\cdot
 \end{align}
\end{theo}
\bi 
Note that the partitions in $\Odd$ and $\E$ are not grounded partitions, but we will see in \Sct{sect:oeground} how to render them as regular partitions.
\subsection{Weighted words, flat partitions and regular partitions}
Let $\C$ be a set of primary colors, and let $\ep$ be  a \textbf{minimal energy} as defined in \Def{def:minerg}.
\subsubsection{Weighted words at degree one} Let us fix a ground $\co\in \C$. We set $\Flt$ to be the set of \textit{primary} flat partitions, which are the flat partitions with ground $\co$ and energy $\ep$. Recall that the energy  $\ep$ defines a relation $\gte$ as follows, 
\begin{equation}\label{eq:flatcond}
k_{c}\gte k'_{c'} \Longleftrightarrow k-k'= \ep(c,c')\cdot
\end{equation}
\m Let us also recall the energy relation $\od$ defined by  
\begin{equation}\label{eq:diffcond}
k_{c}\od k'_{c'} \Longleftrightarrow k-k'\geq \ep(c,c')\,,
\end{equation}
and let $\Reg$ be the set of \textit{primary} regular partitions, which are the regular partitions in ground $\co$ and with energy $\ep$. 
\m
Assuming that $\co = 1$, one can see, for both flat or regular partitions, the color sequence as a product of colors in $\C\setminus \{\co\}$. Let us set $\C'=\C\setminus \{\co\}$. Depending on certain properties of $\ep$,
we can build a bijection between $\Reg$ and $\Flt$ which preserves both the size and the color sequence of partitions.
\bi 
\begin{theo}[degree one]\label{theo:flatreg1} 
Let $\delta_g\in \{0,1\}$. 
Assume that $\ep(\co,\co)=0$, and that for all $c\neq 0$, 
 \begin{equation}\label{eq:pos0}
 \delta_g=\ep(\co,c)=1-\ep(c,\co)\,\cdot
 \end{equation}
 There then exists a bijection $\Omega$ between $\Flt$ and $\Reg$ which preserves the total energy and the sequence of colors different from $\co$.
\end{theo}
\bi 
This theorem is a generalization of \Thm{theo:kx}. To see that \Thm{theo:flatreg1} implies \Thm{theo:kx}, we take the set $\C=\{c_0,\ldots,c_{m-1}\}$, and set $\co=c_0$.
\Thm{theo:kx} then corresponds to the energy $\ep(c_i,c_j)= \chi(i<j)$, 
followed by the transformation 
$$(q,c_0,c_1,\ldots,c_{m-1})\mapsto (q^m,1,q,\ldots,q^{m-1})\,\cdot$$
The latter operation means that the part is 
$k_{c_i}$ is transformed into the part $mk+i$, and the relations in \eqref{eq:flatcond} and  \eqref{eq:diffcond} induced by $\ep$ then become 
$$
mk+i \gte mk'+i' \Longleftrightarrow (mk+i) -(mk'+i') = \begin{cases}
i-i' \ \ \text{if} \ \ i\geq i'\\
m+i-i'\ \text{if} \ \ i< i'
\end{cases}\,,
$$
$$
mk+i \od mk'+i' \Longleftrightarrow (mk+i) -(mk'+i') \geq  \begin{cases}
i-i' \ \ \text{if} \ \ i\geq i'\\
m+i-i'\ \text{if} \ \ i< i'
\end{cases}\,\cdot
$$
Note that the last part corresponds to $0$ for both flat and regular partitions after this transformation. We then retrieve the flat partitions of \Expl{ex:mflat} and the regular partitions in \Expl{ex:mreg}, except that we implicitly assimilate the congruence modulo $m$ of the part size to the unique corresponding  color in $\C$.
\m Similarly, \Thm{theo:kx} is also implied by \Thm{theo:flatreg1} with the energy $\ep(c_i,c_j)= \chi(i>j)$
followed 
by the transformation 
$(q,c_0,c_1,\ldots,c_{m-1})\mapsto (q^m,1,q^{-1},\ldots,q^{1-m})$, in which case the part $k_{c_i}$ is assimilated to the part $km-i$.
\m In the same way, we obtain the analogue of Glaisher, stated in \Cor{cor:analogueglaisher}, by considering the same set of colors $\C=\{c_0,\ldots,c_{m-1}\}$, the ground $\co=c_0$, the transformation $(q,c_0,c_1,\ldots,c_{m-1})\mapsto (q^m,1,q,\ldots,q^{m-1})$, but a slightly different energy  $\ep$, given in \Expl{ex:ordered1} with $I_2=\{0\}$, 
\[
\ep(c_i,c_j)=\begin{cases}
\chi(i<j) \ \ \text{if} \ \ i\neq j\\
0\ \ \text{if} \ \ i= j=0\\
1\ \ \text{if} \ \ i= j\neq 0
\end{cases}\,\cdot
\]
Note that the restriction of $\ep$ to $\C\setminus \{c_0\}=\C'$ then gives $\ep(c_i,c_j)= \chi(i\leq j)$.
\subsubsection{Weighted words at degree two} Let us now assume that $\ep$ satisfies the conditions of \Thm{theo:flatreg1} and consider the set of secondary parts $\Sc$ defined in \Def{def:secpar}. Recall that $\delta_g$ is the common value of $\ep(\co,c)$ for all $c\in \C'$.
\begin{deff}\label{def:flat2}
We define $\Fltt$ to be the set of \textit{secondary} flat partitions, which are the flat partitions into secondary parts in $\Sc$, with  ground $\co^2$ and energy $\ep_2$ defined by 
\begin{equation}
\ep_2(cc',dd') = \ep(c,c')+2\ep(c',d)+\ep(d,d')\,\cdot
\end{equation}
\end{deff}
\begin{rem}
The definition of $\ep_2$ equivalent to defining a relation $\gtei$ on secondary parts which satisfies the following:
\begin{align}\label{eq:flatcond2}
(2k+\ep(c,c'))_{cc'}\gtei (2l+\ep(d,d'))_{dd'} &\Longleftrightarrow (2k+\ep(c,c'))-(2l+\ep(d,d'))= \ep(c,c')+2\ep(c',d)+\ep(d,d')\,\cdot\nonumber\\
&\Longleftrightarrow k-(l+\ep(d,d'))= \ep(c',d)\nonumber\\
&\Longleftrightarrow \mu((2k+\ep(c,c'))_{cc'})\gte \gamma((2l+\ep(d,d'))_{dd'})\,\cdot
\end{align}  
\end{rem}
\begin{deff}
We set $\Regg$ to be the set of \textit{secondary} regular partitions, which are the regular partitions into secondary parts in $\Sc$, with  ground $\co^2$ and the energy $\ep'$ defined by
\begin{equation}
\ep'_2(cc',dd') = \ep_2(cc',dd')+2\delta^\ep(cc',dd')\,,
\end{equation}
where $\delta^\ep(cc',dd')=0$ apart from
\begin{equation}\label{eq:always}
\delta^\ep(c\co,\co d')=\ep(c,d')\quad\text{for all}\quad c,d'\in \C'\,,
\end{equation}
and the additional exceptions when $\delta_g=1$:
\begin{align}
\delta^\ep(cc',dd')&=-1 \quad\text{if}\quad \begin{cases} c=\co, \ \ c',d,d'\in \C' \ \ \text{and}\ \ \ep(c',d)=1\\
c'=\co, \ \ c,d,d'\in \C' \ \ \text{and}\ \ \ep(c,d)=0\\
\end{cases}\label{eq:-1}\\
\delta^\ep(cc',dd')&=1 \quad\text{if}\quad \begin{cases} d'=\co, \ \ c',d\in \C' \ \ \text{and}\ \ \ep(c',d)=0\\
d=\co, \ \ c,c',d'\in \C' \ \ \text{and}\ \ \ep(c',d')=1
\end{cases}\label{eq:1}\,\cdot
\end{align}
\end{deff}
\begin{rem}
Note that the energy $\ep'_2$ defines a binary relation $\odg$ on secondary parts of $\Sc$ as follows,
\begin{equation}\label{eq:diffcd2}
(2k+\ep(c,c'))_{cc'}\odg (2l+\ep(d,d'))_{dd'} \Longleftrightarrow k-l-\ep(c',d)-\ep(d,d')\geq \delta^\ep(cc',dd')\,\cdot
\end{equation}
\end{rem}
\bi 
The level above \Thm{theo:flatreg1}  can be stated as follows.
\begin{theo}[degree two]\label{theo:flatreg2} Assuming that $\co=1$,
 there exists a bijection between $\Regg$ and $\Fltt$ which preserves the total energy and the sequence of colors different from $\co$.
\end{theo}
Let us give a example of such identity. Consider the set $\C=\{a,b,c\}$, $\co=c$ and the energy matrix
$$M_\ep = 
\bordermatrix{
\text{}&a&b&c
\cr a&1&1&1
\cr b&0&1&1
\cr c&0&0&0
}\,\cdot
$$
We then obtain the energy matrices for $\ep_2$ and $\ep'_2$
$$
\begin{footnotesize}
M_{\ep_2} = 
\bordermatrix{
\text{}&a^2&ab&ac&ba&b^2&bc&ca&cb&c^2
\cr a^2&4&4&4&3&4&4&3&3&3
\cr ab&2&2&2&3&4&4&3&3&3
\cr ac&2&2&2&1&2&2&1&1&1
\cr ba&3&3&3&2&3&3&2&2&2
\cr b^2&2&2&2&3&4&4&3&3&3
\cr bc&2&2&2&1&2&2&1&1&1
\cr ca&3&3&3&2&3&3&2&2&2
\cr cb&1&1&1&2&3&3&2&2&2
\cr c^2&1&1&1&0&1&1&0&0&0
}\quad,\qquad M_{\ep'_2} = 
\bordermatrix{
\text{}&a^2&ab&ac&ba&b^2&bc&ca&cb&c^2
\cr a^2&4&4&4&3&4&4&3&3&3
\cr ab&2&2&2&3&4&4&3&3&3
\cr ac&2&2&2&1&2&2&3&3&1
\cr ba&3&3&3&2&3&3&2&2&2
\cr b^2&2&2&2&3&4&4&3&3&3
\cr bc&2&2&2&1&2&2&1&3&1
\cr ca&3&3&3&2&3&3&2&2&2
\cr cb&1&1&1&2&3&3&2&2&2
\cr c^2&1&1&1&0&1&1&0&0&0
}\,\cdot
\end{footnotesize}
$$
Since in the regular partitions we never have a color $b^2$ except for the last part $0_{b^2}$, one can consider partitions into parts with color in $\{a^2,ab,ac,ba,b^2,bc,ca,cb\}$, satisfying the  minimal difference condition in 
\[
\qquad M_{\ep'_2} = 
\bordermatrix{
\text{}&a^2&ab&ac&ba&b^2&bc&ca&cb
\cr a^2&4&4&4&3&4&4&3&3
\cr ab&2&2&2&3&4&4&3&3
\cr ac&2&2&2&1&2&2&3&3
\cr ba&3&3&3&2&3&3&2&2
\cr b^2&2&2&2&3&4&4&3&3
\cr bc&2&2&2&1&2&2&1&3
\cr ca&3&3&3&2&3&3&2&2
\cr cb&1&1&1&2&3&3&2&2
}\,
\]
and such that the minimal sizes for the part with color $a^2,ab,ac,ba,b^2,bc,ca,cb$ are respectively $3,3,1,2,3,1,2,2$. By applying the transformation $(q,a,b,c)\mapsto (q^4,q^{-3},q^{-1},1)$, we obtain the companion of Siladic's identity given in Corollary \ref{cor:silcomp}. 
\section{Applications to level one perfect crystals}\label{sect:level1}
\subsection{Notion of crystals}
\subsubsection{Crystals}
We here introduce the basic tools which will be useful for the computation of level one standard modules' characters. For further references, see \cite{HK02,(KMN)$^2$a}.
\bi 
Let $n$ be a non-negative integer, and consider the Cartan datum $(A,\Pi,\Pi^\vee,P,P^\vee)$ for a generalised Cartan matrix $A$ of affine type and rank $n$.
The set $\Pi$ is the set of the simple roots $\alpha_i (i \in \{0, \dots , n\})$, and 
we denote by $\bar P = \Z\Lambda_0 \oplus \cdots \oplus \Z\Lambda_{n}$ the lattice of the classical weights, where the elements $\Lambda_{\ell}$ $(\ell \in \{0, \dots , n\})$ are the 
fundamental weights. The null root $\delta$, can be uniquely written 
$$\delta = d_0\alpha_0+\cdots+d_n\alpha_n$$
with $d_0,\ldots,d_n$ being positive integers.
Let us now introduce the notion of crystal.
\begin{deff}
Let $A=(a_{i,j})_{0 \leq i,j \leq n-1}$ be a Cartan matrix with associated \textit{Cartan datum} $(A,\Pi,\Pi^\vee,P,P^\vee).$ A \textit{crystal} associated with $(A,\Pi,\Pi^\vee,P,P^\vee)$ is a set $\B$ together with maps
\begin{align*}
\mathrm{wt} : \B &\longrightarrow P,\\
\tilde e_i, \tilde f_i : \B &\longrightarrow \B \cup \{0\} \qquad \quad\ (i \in \I),\\
\varepsilon_i,\varphi_i : \B &\longrightarrow \Z \cup \{- \infty\} \qquad (i \in \I),
\end{align*}
satisfying the following properties for all $i \in \I$:
\begin{enumerate}
\item $\varphi_i(b)=\varepsilon_i(b)+ \langle h_i, \mathrm{wt} (b) \rangle,$
\item $\mathrm{wt}(\tilde e_ib)= \mathrm{wt}b+\alpha_i$ if $\tilde{e_i}b \in \B,$
\item $\mathrm{wt}(\tilde f_ib)= \mathrm{wt}b-\alpha_i$ if $\tilde{f_i}b \in \B,$
\item $\varepsilon_i(\tilde e_ib)= \varepsilon_i(b)-1$ if $\tilde{e_i}b \in \B,$
\item  $\varphi_i(\tilde e_ib)= \varphi_i(b)+1$ if $\tilde{e_i}b \in \B,$
\item $\varepsilon_i(\tilde f_ib)= \varepsilon_i(b)+1$ if $\tilde{f_i}b \in \B,$
\item  $\varphi_i(\tilde f_ib)= \varphi_i(b)-1$ if $\tilde{f_i}b \in \B,$
\item $\tilde{f_i}b=b'$ if and only if $b= \tilde{e_i}b'$ for $b,b' \in \B,$
\item if $\varphi_i(b)= - \infty$ for $b \in \B,$ then $\tilde{e_i}b= \tilde{f_i}b=0.$
\end{enumerate}
\end{deff}
A graphical representation of a crystal $\B$, called the crystal graph, consists of a graph whose vertices are the elements of $\B$, and whose edges are $i$-arrows satisfying
$$ b \xrightarrow[]{\,\,\, i \,\,\,} b' \quad \text{if and only if}\quad \tilde f_i b = b' \text{ (or equivalently } \tilde e_i b' = b) .$$
In the following, for $i \in \I$,  we define the functions $\varepsilon_i, \varphi_i: \B \rightarrow \Z$ by
$$
\begin{array}{cc}
&\varepsilon_i(b) = \max\{ k  \geq 0 \mid \tilde e_i^k b \in \B\}, \\
&\varphi_i(b) =  \max\{ k  \geq 0 \mid \tilde f_i^k b \in \B\}. \end{array}
$$
In other words,  $\varepsilon_i(b)$ is the length of the longest chain of $i$-arrows ending at $b$ in the crystal graph, and $\varphi_i(b)$ is the length of the longest chain of
$i$-arrows starting from $b$. These definitions for $\varepsilon_i$ and $\varphi_i$ will be possible because of the nature of the crystals coming from the crystal base of integrable modules that we will consider in the following. Furthermore, by setting
\begin{equation}
\label{eq:epsilonphi}
\varepsilon(b) = \sum_{i=0}^{n-1} \varepsilon_i(b) \Lambda_i,
\qquad \text{and} \qquad \varphi(b) = \sum_{i=0}^{n-1} \varphi_i(b) \Lambda_i,
\end{equation}
we then have $ \wt(b) = \varphi(b) -\varepsilon(b)$ for all $b \in \B$, where $\wt(b)$ is the projection of $\mathrm{wt}(b)$ on $\overline{P}$.
\bi  We now define the tensor product of crystals.
\begin{deff}
Let $\B_1, \B_2$ two crystals associated with $(A,\Pi,\Pi^\vee,P,P^\vee)$. The tensor product $\B
= \B_1 \ot \B_2 \equiv \B_1 \times \B_2$ is the crystal satisfying the following: 
\begin{equation}\label{eq:tensrul}
\begin{array}{cc}
&\tilde e_i (b_1 \ot b_2) = \begin{cases} \tilde e_i b_1 \ot b_2 \quad & \hbox{\rm if} \ \ \varphi_i(b_1)  \geq \varepsilon_i(b_2), \\
b_1 \ot \tilde e_i b_2 \quad & \hbox{\rm if} \ \ \varphi_i(b_1) < \varepsilon_i(b_2), \end{cases}\\
&\tilde f_i (b_1 \ot b_2) = \begin{cases} \tilde f_i b_1 \ot b_2 \quad & \hbox{\rm if} \ \ \varphi_i(b_1) > \varepsilon_i(b_2), \\
b_1 \ot \tilde f_i b_2 \quad & \hbox{\rm if} \ \ \varphi_i(b_1) \leq \varepsilon_i(b_2), \end{cases}
\end{array} \end{equation}
where $b_1 \ot 0 = 0\ot b_2 = 0$ for all $b_1\in\B_1$ and $b_2\in \B_2$, and 
\begin{align*}
 \mathrm{wt}(b_1\ot b_2)&= \mathrm{wt} b_1+ \mathrm{wt} b_2,\\
 \varepsilon_i (b_1\ot b_2) &= \max\{\varepsilon_i(b_1),\varepsilon_i(b_1)+\varepsilon_i(b_2)-\varphi_i(b_1)\},\\
 \varphi_i (b_1\ot b_2) &= \max\{\varphi_i(b_2),\varphi_i(b_1)+\varphi_i(b_2)-\varepsilon_i(b_2)\}.
\end{align*}
\end{deff}
To fully understand the tensor rule, we can picture it on the crystal graph with the following maximal chains of $i$-arrows and $j$-arrows with $i\neq j$.
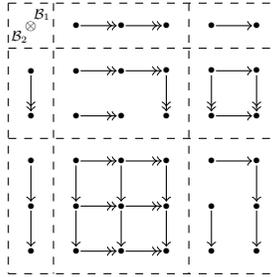
\begin{center}
\begin{tikzpicture}[scale=0.6, every node/.style={scale=0.6}]

\draw (0.25,0.25) node {$\B_1$};
\draw (0,0) node {$\ot$};
\draw (-0.25,-0.25) node {$\B_2$};

\foreach \x in {0,...,5}
\foreach \y in {1,...,5}
\draw (\x,-\y) node {$\bullet$};;

\foreach \x in {1,...,5}
\draw (\x,0) node {$\bullet$};

\draw [dashed] (-0.5,-2.5)--(5.5,-2.5);
\draw [dashed] (-0.5,-0.5)--(5.5,-0.5);
\draw [dashed] (-0.5,0.5)--(5.5,0.5);
\draw [dashed] (-0.5,-5.5)--(5.5,-5.5);

\draw [dashed] (0.5,0.5)--(0.5,-5.5);
\draw [dashed] (3.5,0.5)--(3.5,-5.5);
\draw [dashed] (-0.5,0.5)--(-0.5,-5.5);
\draw [dashed] (5.5,0.5)--(5.5,-5.5);

\foreach \x in {0,3,4,5}
\draw [->>] (\x,-1-0.1) -- (\x,-2+0.1);

\foreach \y in {0,1,2,3}
\draw [->] (4+0.1,-\y) -- (5-0.1,-\y);

\foreach \x in {0,1,2,3,5}
\draw [->] (\x,-3-0.1) -- (\x,-4+0.1);
\foreach \x in {0,1,2,3,4,5}
\draw [->] (\x,-4-0.1) -- (\x,-5+0.1);

\foreach \y in {0,1,2,3,4,5}
\draw [->>] (1+0.1,-\y) -- (2-0.1,-\y);
\foreach \y in {0,1,3,4,5}
\draw [->>] (2+0.1,-\y) -- (3-0.1,-\y);

\end{tikzpicture}
\end{center}
An important property of the tensor product is its associativity:  $(\B_1\ot \B_2)\ot \B_3 = \B_1 \ot (\B_2\ot \B_3)$.
\bi  
We finally introduce the notion of energy function. 
\begin{deff} Let $\B$ be crystal base.
An {\it energy function} on $\B \ot \B$ is a map $H: \B \ot \B \rightarrow
\Z$ satisfying
\begin{equation}\label{eq:ef} H\left(\tilde e_i (b_1 \ot b_2)\right)
= \begin{cases} 
H(b_1 \ot b_2) + \chi(i=0) & \qquad \hbox{\rm if} \ \  \varphi_i(b_1)  \geq \varepsilon_i(b_2)  \\
H(b_1 \ot b_2) - \chi(i=0) & \qquad \hbox{\rm if} \ \  \varphi_i(b_1)< \varepsilon_i(b_2),
\end{cases} \end{equation}
for all $i \in \I$ and $b_1,b_2$ with $\tilde e(b_1 \ot b_2) \neq 0$.
\end{deff}
By definition, in the crystal graph of $\B \ot \B$, the value of $H(b_1\ot b_2)$, when it exists, determines all the values $H(b'_1\ot b'_2)$ for vertices $b'_1\ot b'_2$ in the same connected component as $b_1\ot b_2$.
\subsubsection{Character formula}
The notion of perfect crystals, introduced by Kang and al. \cite{(KMN)$^2$a,(KMN)$^2$b}, appears as a possible method to compute the characters of standard modules. The notion of grounded partitions, introduced by the author and Dousse, was deeply influenced by the behaviour of the perfect crystals and the related character formula.
\bi 
Let us consider a perfect crystal $\B$ of level $\ell$, a classical weight $\Ll$  of level $\ell$ satisfying  $\p_{\Lambda}=(\cdots\ot g \ot g)$, where $\p_{\Lambda}$ is the ground state path, and let us assume that $ H(g \ot g)=0$. We then define the set of colors indexed by $\B$
$$\C_{\B}=\{c_b: \,b\in \B\}$$
and the energy $\ep$ by 
\begin{equation}\label{eq:restr}
\ep(c_b,c_{b'})=H(b'\ot b)\cdot
\end{equation}
We then obtain the following theorem from \cite{DK2}.
\begin{theo}[Dousse-K.] \label{theo:formchar}
By a change of variable $q=e^{-\delta/d_0}$ and $c_b=e^{\overline{\mathrm{wt}} b}$ for all $b\in \B$, we have $\co=1$ and the following identity:
\begin{equation}
 \sum_{\pi\in \Flt} C(\pi)q^{|\pi|} = e^{-\Ll}\rm{ch}(L(\Ll))\,,
\end{equation}
where $\Flt$ is the set of flat partitions with ground $\co$ and energy $\ep$, and  we assume that the colors $c_b$ commute in the generating function. 
\footnote{Here we use a corrected version of  KMN character formula, with $\delta/d_0$ instead of $\delta$.}
\end{theo}
In the remainder of this section, by using the above theorem and \Thm{theo:flatreg1}, we compute the character formula corresponding to the following level one weights:
\begin{itemize}
\item $\Lambda_0$ for the affine type $A_{2n}^{(2)}(n\geq 2)$,
\item $\Lambda_0$ and $\Lambda_n$ for the affine type $D_{n+1}^{(2)}(n\geq 2)$,
\item $\Lambda_n$ for the affine type $B_{n}^{(1)}(n\geq 3)$.
\end{itemize}
\subsection{Case of affine type $A_{2n}^{(2)}(n\geq 2)$}
The crystal $\B$ of the vector representation of $A_{2n}^{(2)}(n\geq 2)$ is given by the crystal graph below
\begin{fig}
\[\]
\begin{center}
\begin{tikzpicture}[scale=0.6, every node/.style={scale=0.6}]
\draw (-4,0) node{$\B$ :};
\draw (-4,-0.75) node{$b^{\Ll_0}=b_{\Ll_0}=0$};
\draw (-4,-1.5) node{$\p_{\Ll_0}=(\cdots000)$};
\draw (-1.5,-0.75) node {$0$};
\draw (0,0) node {$1$};
\draw (1.5,0) node {$2$};
\draw (5,0) node {$n-1$};
\draw (6.75,0) node {$n$};
\draw (3.15,0) node {$\cdots$};
\draw (0,-1.5) node {$\overline{1}$};
\draw (1.5,-1.5) node {$\overline{2}$};
\draw (5,-1.5) node {$\overline{n-1}$};
\draw (6.75,-1.5) node {$\overline{n}$};
\draw (3.15,-1.5) node {$\cdots$};

\draw (-0.2-1.5,0.2-0.75)--(0.2-1.5,0.2-0.75)--(0.2-1.5,-0.2-0.75)--(-0.2-1.5,-0.2-0.75)--(-0.2-1.5,0.2-0.75);

\foreach \x in {0,1,4.5} 
\foreach \y in {0,1}
\draw (-0.2+1.5*\x,0.2-1.5*\y)--(0.2+1.5*\x,0.2-1.5*\y)--(0.2+1.5*\x,-0.2-1.5*\y)--(-0.2+1.5*\x,-0.2-1.5*\y)--(-0.2+1.5*\x,0.2-1.5*\y);

\foreach \x in {0} 
\foreach \y  in {0,1}
\draw (-0.45+5+2*\x,0.2-1.5*\y)--(0.45+5+2*\x,0.2-1.5*\y)--(0.45+5+2*\x,-0.2-1.5*\y)--(-0.45+5+2*\x,-0.2-1.5*\y)--(-0.45+5+2*\x,0.2-1.5*\y);

\draw [thick, ->] (-1.45,-0.5)--(-0.3,-0.05);
\draw [thick, <-] (-1.45,-1)--(-0.3,-1.45);
\draw (-0.8,-0.1) node {\footnotesize{$0$}};
\draw (-0.8,-1.4) node {\footnotesize{$0$}};

\draw [thick, ->] (0.3,0)--(1.2,0);
\draw [thick, <-] (0.3,-1.5)--(1.2,-1.5);
\draw (0.75,0.15) node {\footnotesize{$1$}};
\draw (0.75,-1.67) node {\footnotesize{$1$}};

\draw [thick, ->] (1.8,0)--(2.7,0);
\draw [thick, <-] (1.8,-1.5)--(2.7,-1.5);
\draw (2.25,0.15) node {\footnotesize{$2$}};
\draw (2.25,-1.67) node {\footnotesize{$2$}};

\draw [thick, ->] (3.6,0)--(4.5,0);
\draw [thick, <-] (3.6,-1.5)--(4.5,-1.5);
\draw (4.05,0.15) node {\footnotesize{$n-2$}};
\draw (4.05,-1.67) node {\footnotesize{$n-2$}};

\draw [thick, ->] (5.6,0)--(6.4,0);
\draw [thick, <-] (5.6,-1.5)--(6.4,-1.5);
\draw (6,0.15) node {\footnotesize{$n-1$}};
\draw (6,-1.67) node {\footnotesize{$n-1$}};

\draw [thick, ->] (6.75,-0.3)--(6.75,-1.2);
\draw (6.9,-0.75) node {\footnotesize{$n$}};

\end{tikzpicture}
\end{center}
\end{fig}
with $\wt (0)=0$ and for all $u\in \{1,\ldots,n\}$,
\begin{equation}\label{eq:weightA}
-\wt (\overline{u}) = \wt u = \frac{1}{2}\alpha_n+\sum_{i=u}^{n-1} \alpha_i\,\cdot
\end{equation} 
Here, we have 
$\delta = \alpha_n+2\sum_{i=0}^{n-1} \alpha_i$.
We thus obtain the following crystal graph for $\B\ot \B$ 
\begin{fig}
\[\]
\begin{center}
\begin{tikzpicture}[scale=0.5, every node/.style={scale=0.5}]

\draw [thick,->] (-10,-1)--(-9,-1);
\draw (-8.8,-1) node[right] {: $0$-arrow};

\draw [->] (-10,-2)--(-9,-2);
\draw (-8.8,-2) node[right] {: $n$-arrow};

\draw [dashed,->] (-10,-3)--(-9,-3);
\draw (-8.8,-3) node[right] {: paths of $i$-arrows, for consecutive $i \neq 0,n$};

\draw (-10,-3.6)--(-9,-3.6)--(-9,-4.4)--(-10,-4.4)--(-10,-3.6);
\draw (-8.8,-4) node[right] {: connected components without $0$-arrows};

 


\foreach \x in {0,...,7} 
\foreach \y in {0,...,7}
\draw (-0.2-0.4+3*\x,0.2-1.5*\y)--(0.2-0.4+3*\x,0.2-1.5*\y)--(0.2-0.4+3*\x,-0.2-1.5*\y)--(-0.2-0.4+3*\x,-0.2-1.5*\y)--(-0.2-0.4+3*\x,0.2-1.5*\y);

\foreach \x in {0,...,7} 
\foreach \y in {0,...,7}
\draw (-0.2+0.4+3*\x,0.2-1.5*\y)--(0.2+0.4+3*\x,0.2-1.5*\y)--(0.2+0.4+3*\x,-0.2-1.5*\y)--(-0.2+0.4+3*\x,-0.2-1.5*\y)--(-0.2+0.4+3*\x,0.2-1.5*\y);

\foreach \x in {0,...,7} 
\foreach \y in {0,...,7}
\draw (3*\x,-1.5*\y) node {$\ot$};


\foreach \x in {0,...,7} 
\draw (0.4+3*\x,0) node {$0$};
\foreach \y in {0,...,7} 
\draw (-0.4,-1.5*\y) node {$0$};

\foreach \x in {0,...,7} 
\draw (0.4+3*\x,-1.5) node {$1$};
\foreach \y in {0,...,7} 
\draw (-0.4+3,-1.5*\y) node {$1$};

\foreach \x in {0,...,7} 
\draw (0.4+3*\x,-3) node {$u$};
\foreach \y in {0,...,7} 
\draw (-0.4+6,-1.5*\y) node {$u$};

\foreach \x in {0,...,7} 
\draw (0.4+3*\x,-4.5) node {$n$};
\foreach \y in {0,...,7} 
\draw (-0.4+9,-1.5*\y) node {$n$};

\foreach \x in {0,...,7} 
\draw (0.4+3*\x,-6) node {$\overline{n}$};
\foreach \y in {0,...,7} 
\draw (-0.4+12,-1.5*\y) node {$\overline{n}$};

\foreach \x in {0,...,7} 
\draw (0.4+3*\x,-7.5) node {$\overline{u}$};
\foreach \y in {0,...,7} 
\draw (-0.4+15,-1.5*\y) node {$\overline{u}$};

\foreach \x in {0,...,7} 
\draw (0.4+3*\x,-9) node {$\overline{1}$};
\foreach \y in {0,...,7} 
\draw (-0.4+18,-1.5*\y) node {$\overline{1}$};

\foreach \x in {0,...,7} 
\draw (0.4+3*\x,-10.5) node {$0$};
\foreach \y in {0,...,7} 
\draw (-0.4+21,-1.5*\y) node {$0$};

\foreach \x in {0,...,5,7} 
\draw [thick, ->] (3*\x,-0.25)--(3*\x,-1.25);
\foreach \y in {0,2,3,4,5,6,7} 
\draw [thick, ->] (18.9,-1.5*\y)--(20.1,-1.5*\y);

\foreach \x in {0,2,3,4,6,7} 
\draw [dashed, ->] (3*\x,-1.75)--(3*\x,-2.75);
\foreach \y in {0,1,3,4,5,7} 
\draw [dashed, ->] (15.9,-1.5*\y)--(17.1,-1.5*\y);

\foreach \x in {0,1,3,5,6,7} 
\draw [dashed, ->] (3*\x,-3.25)--(3*\x,-4.25);
\foreach \y in {0,1,2,4,6,7} 
\draw [dashed, ->] (12.9,-1.5*\y)--(14.1,-1.5*\y);

\foreach \x in {0,1,2,4,5,6,7} 
\draw [->] (3*\x,-4.75)--(3*\x,-5.75);
\foreach \y in {0,1,2,3,5,6,7} 
\draw [->] (9.9,-1.5*\y)--(11.1,-1.5*\y);

\foreach \x in {0,2,3,4,6,7}
\draw [dashed, ->] (3*\x,-7.75)--(3*\x,-8.75);
\foreach \y in {0,1,2,4,6,7} 
\draw [dashed, ->] (6.9,-1.5*\y)--(8.1,-1.5*\y);

\foreach \x in {0,1,3,5,6,7} 
\draw [dashed, ->] (3*\x,-6.25)--(3*\x,-7.25);
\foreach \y in {0,1,3,4,5,7} 
\draw [dashed, ->] (3.9,-1.5*\y)--(5.1,-1.5*\y);

\foreach \x in {1,...,5} 
\draw [thick, ->] (3*\x,-9.25)--(3*\x,-10.25);
\foreach \y in {2,...,6} 
\draw [thick, ->] (0.9,-1.5*\y)--(2.1,-1.5*\y);


\foreach \x in {21}
\foreach \y in {10.5}
\draw [red] (-0.7+\x,0.25-\y)--(0.7+\x,0.25-\y)--(0.7+\x,-0.25-\y)--(-0.7+\x,-0.25-\y)--(-0.7+\x,0.25-\y);

\foreach  \x in {21}
\draw [blue](-0.7+\x,0.25-1.5)--(0.7+\x,0.25-1.5)--(0.7+\x,-0.25-9)--(-0.7+\x,-0.25-9)--(-0.7+\x,0.25-1.5);

\foreach \y in {10.5}
\draw [blue](-0.7+3,0.25-\y)--(0.7+18,0.25-\y)--(0.7+18,-0.25-\y)--(-0.7+3,-0.25-\y)--(-0.7+3,0.25-\y);

\draw [red] (-0.7+3,1-3)--(-0.7+3,-0.25-9)--(0.25+16.5,-0.25-9)--(-0.7+3,1-3);

\draw [foge] (-0.7+3,0.25-1.5)--(0.7+18,0.25-1.5)--(0.7+18,-0.25-9)--(-0.7+18,-0.25-9)--(-0.7+3,-0.25-1.5)--(-0.7+3,0.25-1.5);

\end{tikzpicture}
\end{center}
\end{fig}
We then consider the set of colors $\C=\{c_1,\ldots,c_n,c_{\overline{n}},\ldots,c_{\overline{1}},c_0\}$, $\co=c_0$, and by setting $\ep'(c_u,c_v)=H(v\ot u)$ and $H(0\ot 0)=0$, we obtain the following energy matrix for $\ep'$:
\[
\bordermatrix{
\text{}&c_1&\cdots&c_{\overline{1}}&c_0
\cr c_1&2&\cdots&2&1
\cr \vdots&\vdots&\ddots&\vdots&\vdots
\cr c_{\overline{1}}&0&\cdots&2&1
\cr c_0&1&\cdots&1&0
}\,\cdot
\]
This energy matrix can be obtain by taking the energy matrix of $\ep$ defined by
\[
\bordermatrix{
\text{}&c_1&\cdots&c_{\overline{1}}&c_0
\cr c_1&1&\cdots&1&1
\cr \vdots&\vdots&\ddots&\vdots&\vdots
\cr c_{\overline{1}}&0&\cdots&1&1
\cr c_0&0&\cdots&0&0
}
\]
followed by the transformation 
\begin{equation}\label{eq:transfo1}
(q,c_1,c_{\overline{1}},\ldots,c_n,c_{\overline{n}})\mapsto(q^2,c_1q^{-1},c_{\overline{1}}q^{-1},\ldots,c_nq^{-1},c_{\overline{n}}q^{-1})\,\cdot
\end{equation}
This means that, for $c\neq c_0$, the part $k_{c}$ for the energy $\ep$  is identified as the part $(2k-1)_c$ for the energy $\ep'$, and since we do not modify  the ground $c_0$, the part $k_{c_0}$ for $\ep$ is assimilated to $(2k)_{c_0}$ for $\ep'$, so that the last part still remains $0_{c_0}$.
\m
By setting $c_0=1$, we can apply \Thm{theo:flatreg1} to the flat partitions with ground $c_0$ and with energy $\ep$, and we obtain the generation function
\[
\sum_{\pi\in \Flt } C(\pi)q^{|\pi|}=\sum_{\pi\in \Reg } C(\pi)q^{|\pi|} = (-c_1q,-c_{\overline{1}}q,\ldots,-c_nq,-c_{\overline{n}}q;q)_{\infty}\,\cdot
\]
In fact, by the definition of the energy $\ep$, one can view the partitions of $\Reg$ as the finite sub-sequences, ending with $0_{c_0}$, of the infinite sequence 
\[\cdots\od 3_{c_1}\od 2_{c_{\overline{1}}}\od\cdots\od 2_{c_1}\od 1_{c_{\overline{1}}}\od \cdots\od 1_{c_{\overline{n}}}\od 1_{c_n}\od \cdots\od 1_{c_1}\od 0_{c_0}\cdot\]
Using \eqref{eq:transfo1}, we then have that the flat partitions with ground $\co$ and energy $\ep'$ are  generated by the function
\[
(-c_1q,-c_{\overline{1}}q,\ldots,-c_nq,-c_{\overline{n}}q;q^2)_{\infty}\,\cdot
\]
Using \Thm{theo:formchar} and \eqref{eq:weightA}, we obtain the formula for the character for $\Lambda_0$  given in \Thm{theo:char1}.
\subsection{Case of affine type $D_{n+1}^{(2)} (n\geq 2)$}
The crystal graph of the vector representation $\B$ of  $D_{n+1}^{(1)}(n\geq 2)$ is the following,
\begin{fig}
\[\]
\begin{center}
\begin{tikzpicture}[scale=0.6, every node/.style={scale=0.6}]
\draw (-4,0) node{$\B$ :};
\draw (-4,-0.75) node{$\p_{\Ll_0}=(\cdots0\,0\,0\,0) $};
\draw (-4,-1.5) node{$\p_{\Ll_n}=(\cdots\overline{0}\,\overline{0}\,\overline{0}\,\overline{0})$};
\draw (-1.5,-0.75) node {$0$};
\draw (8.25,-0.75) node {$\overline{0}$};
\draw (0,0) node {$1$};
\draw (1.5,0) node {$2$};
\draw (5,0) node {$n-1$};
\draw (6.75,0) node {$n$};
\draw (3.15,0) node {$\cdots$};
\draw (0,-1.5) node {$\overline{1}$};
\draw (1.5,-1.5) node {$\overline{2}$};
\draw (5,-1.5) node {$\overline{n-1}$};
\draw (6.75,-1.5) node {$\overline{n}$};
\draw (3.15,-1.5) node {$\cdots$};
\draw (-0.2-1.5,0.2-0.75)--(0.2-1.5,0.2-0.75)--(0.2-1.5,-0.2-0.75)--(-0.2-1.5,-0.2-0.75)--(-0.2-1.5,0.2-0.75);
\draw (-0.2+8.25,0.2-0.75)--(0.2+8.25,0.2-0.75)--(0.2+8.25,-0.2-0.75)--(-0.2+8.25,-0.2-0.75)--(-0.2+8.25,0.2-0.75);
\foreach \x in {0,1,4.5} 
\foreach \y in {0,1}
\draw (-0.2+1.5*\x,0.2-1.5*\y)--(0.2+1.5*\x,0.2-1.5*\y)--(0.2+1.5*\x,-0.2-1.5*\y)--(-0.2+1.5*\x,-0.2-1.5*\y)--(-0.2+1.5*\x,0.2-1.5*\y);
\foreach \x in {0} 
\foreach \y  in {0,1}
\draw (-0.45+5+2*\x,0.2-1.5*\y)--(0.45+5+2*\x,0.2-1.5*\y)--(0.45+5+2*\x,-0.2-1.5*\y)--(-0.45+5+2*\x,-0.2-1.5*\y)--(-0.45+5+2*\x,0.2-1.5*\y);

\draw [thick, ->] (-1.45,-0.5)--(-0.3,-0.05);
\draw [thick, <-] (-1.45,-1)--(-0.3,-1.45);
\draw (-0.8,-0.1) node {\footnotesize{$0$}};
\draw (-0.8,-1.4) node {\footnotesize{$0$}};
\draw [thick, ->] (0.3,0)--(1.2,0);
\draw [thick, <-] (0.3,-1.5)--(1.2,-1.5);
\draw (0.75,0.15) node {\footnotesize{$1$}};
\draw (0.75,-1.67) node {\footnotesize{$1$}};
\draw [thick, ->] (1.8,0)--(2.7,0);
\draw [thick, <-] (1.8,-1.5)--(2.7,-1.5);
\draw (2.25,0.15) node {\footnotesize{$2$}};
\draw (2.25,-1.67) node {\footnotesize{$2$}};
\draw [thick, ->] (3.6,0)--(4.5,0);
\draw [thick, <-] (3.6,-1.5)--(4.5,-1.5);
\draw (4.05,0.15) node {\footnotesize{$n-2$}};
\draw (4.05,-1.67) node {\footnotesize{$n-2$}};
\draw [thick, ->] (5.6,0)--(6.4,0);
\draw [thick, <-] (5.6,-1.5)--(6.4,-1.5);
\draw (6,0.15) node {\footnotesize{$n-1$}};
\draw (6,-1.67) node {\footnotesize{$n-1$}};
\draw [thick, ->] (7.05,-0.05)--(8.2,-0.5);
\draw [thick, <-] (7.05,-1.45)--(8.2,-1);
\draw (7.6,-0.1) node {\footnotesize{$n$}};
\draw (7.6,-1.4) node {\footnotesize{$n$}};
\end{tikzpicture}
\end{center}
\end{fig}
with $\wt (0)=\wt (\overline{0})=0$ and for all $u\in \{1,\ldots,n\}$,
\begin{equation}\label{eq:weightD}
-\wt (\overline{u}) = \wt u = \sum_{i=u}^{n} \alpha_i\,\cdot
\end{equation}
Here, we have $\delta = \sum_{i=0}^n \alpha_i$.
We thus obtain the following crystal graph for $\B\ot \B$ 
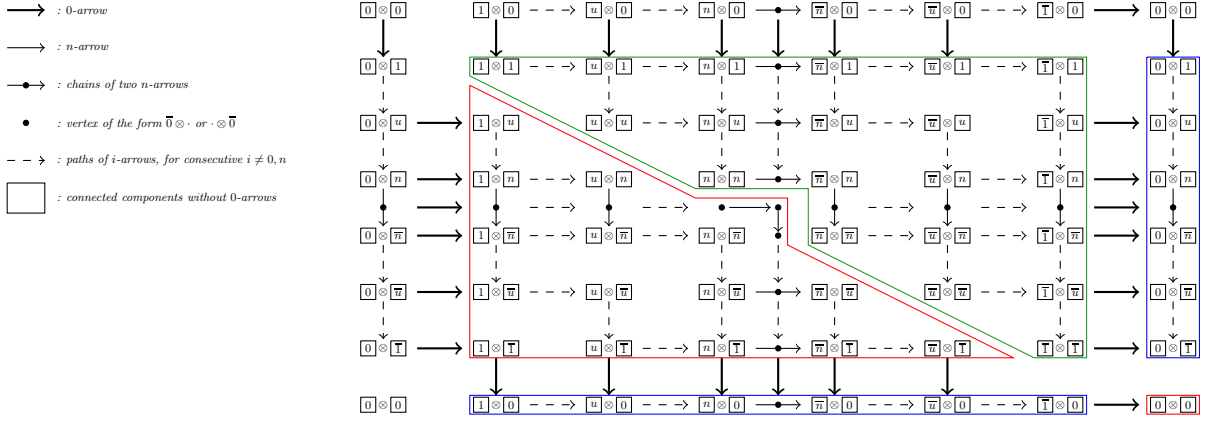
\begin{fig}
\[\]
\begin{center}
\begin{tikzpicture}[scale=0.5, every node/.style={scale=0.5}]
\draw [thick,->] (-10,0)--(-9,0);
\draw (-8.8,0) node[right] {: $0$-arrow};
\draw [->] (-10,-1)--(-9,-1);
\draw (-8.8,-1) node[right] {: $n$-arrow};
\draw [->] (-10,-2)--(-9,-2);
\filldraw (-9.5,-2) circle(2pt);
\draw (-8.8,-2) node[right] {: chains of two $n$-arrows};
\filldraw (-9.5,-3) circle(2pt);
\draw (-8.8,-3) node[right] {: vertex of the form $\overline{0}\ot \cdot$ or $\cdot\ot\overline{0}$};
\draw [dashed,->] (-10,-4)--(-9,-4);
\draw (-8.8,-4) node[right] {: paths of $i$-arrows, for consecutive $i\neq 0,n$};
\draw (-10,-4.6)--(-9,-4.6)--(-9,-5.4)--(-10,-5.4)--(-10,-4.6);
\draw (-8.8,-5) node[right] {: connected components without $0$-arrows};
\foreach \x in {0,...,7} 
\foreach \y in {0,...,7}
\draw (-0.2-0.4+3*\x,0.2-1.5*\y)--(0.2-0.4+3*\x,0.2-1.5*\y)--(0.2-0.4+3*\x,-0.2-1.5*\y)--(-0.2-0.4+3*\x,-0.2-1.5*\y)--(-0.2-0.4+3*\x,0.2-1.5*\y);
\foreach \x in {0,...,7} 
\foreach \y in {0,...,7}
\draw (-0.2+0.4+3*\x,0.2-1.5*\y)--(0.2+0.4+3*\x,0.2-1.5*\y)--(0.2+0.4+3*\x,-0.2-1.5*\y)--(-0.2+0.4+3*\x,-0.2-1.5*\y)--(-0.2+0.4+3*\x,0.2-1.5*\y);

\foreach \x in {0,...,7} 
\foreach \y in {0,...,7}
\draw (3*\x,-1.5*\y) node {$\ot$};

\foreach \x in {0,...,7} 
\draw (0.4+3*\x,0) node {$0$};
\foreach \y in {0,...,7} 
\draw (-0.4,-1.5*\y) node {$0$};

\foreach \x in {0,...,7} 
\draw (0.4+3*\x,-1.5) node {$1$};
\foreach \y in {0,...,7} 
\draw (-0.4+3,-1.5*\y) node {$1$};

\foreach \x in {0,...,7} 
\draw (0.4+3*\x,-3) node {$u$};
\foreach \y in {0,...,7} 
\draw (-0.4+6,-1.5*\y) node {$u$};

\foreach \x in {0,...,7} 
\draw (0.4+3*\x,-4.5) node {$n$};
\foreach \y in {0,...,7} 
\draw (-0.4+9,-1.5*\y) node {$n$};

\foreach \x in {0,...,7} 
\draw (0.4+3*\x,-6) node {$\overline{n}$};
\foreach \y in {0,...,7} 
\draw (-0.4+12,-1.5*\y) node {$\overline{n}$};

\foreach \x in {0,...,7} 
\draw (0.4+3*\x,-7.5) node {$\overline{u}$};
\foreach \y in {0,...,7} 
\draw (-0.4+15,-1.5*\y) node {$\overline{u}$};

\foreach \x in {0,...,7} 
\draw (0.4+3*\x,-9) node {$\overline{1}$};
\foreach \y in {0,...,7} 
\draw (-0.4+18,-1.5*\y) node {$\overline{1}$};

\foreach \x in {0,...,7} 
\draw (0.4+3*\x,-10.5) node {$0$};
\foreach \y in {0,...,7} 
\draw (-0.4+21,-1.5*\y) node {$0$};

\foreach \x in {0,...,5,7} 
\draw [thick, ->] (3*\x,-0.25)--(3*\x,-1.25);
\draw [thick, ->] (3*3.5,-0.25)--(3*3.5,-1.25);
\foreach \y in {0,2,3,4,5,6,7} 
\draw [thick, ->] (18.9,-1.5*\y)--(20.1,-1.5*\y);
\draw [thick, ->] (18.9,-1.5*3.5)--(20.1,-1.5*3.5);

\foreach \x in {0,2,3,4,6,7} 
\draw [dashed, ->] (3*\x,-1.75)--(3*\x,-2.75);
\draw [dashed, ->] (3*3.5,-1.75)--(3*3.5,-2.75);
\foreach \y in {0,1,3,4,5,7} 
\draw [dashed, ->] (15.9,-1.5*\y)--(17.1,-1.5*\y);
\draw [dashed, ->] (15.9,-1.5*3.5)--(17.1,-1.5*3.5);

\foreach \x in {0,1,3,5,6,7} 
\draw [dashed, ->] (3*\x,-3.25)--(3*\x,-4.25);
\draw [dashed, ->] (3*3.5,-3.25)--(3*3.5,-4.25);
\foreach \y in {0,1,2,4,6,7} 
\draw [dashed, ->] (12.9,-1.5*\y)--(14.1,-1.5*\y);
\draw [dashed, ->] (12.9,-1.5*3.5)--(14.1,-1.5*3.5);

\foreach \x in {0,1,2,4,5,6,7} 
\draw [->] (3*\x,-4.75)--(3*\x,-5.75);
\foreach \y in {0,1,2,3,5,6,7} 
\draw [->] (9.9,-1.5*\y)--(11.1,-1.5*\y);

\foreach \x in {0,2,3,4,6,7}
\draw [dashed, ->] (3*\x,-7.75)--(3*\x,-8.75);
\draw [dashed, ->] (3*3.5,-7.75)--(3*3.5,-8.75);
\foreach \y in {0,1,2,4,6,7} 
\draw [dashed, ->] (6.9,-1.5*\y)--(8.1,-1.5*\y);
\draw [dashed, ->] (6.9,-1.5*3.5)--(8.1,-1.5*3.5);

\foreach \x in {0,1,3,5,6,7} 
\draw [dashed, ->] (3*\x,-6.25)--(3*\x,-7.25);
\draw [dashed, ->] (3*3.5,-6.25)--(3*3.5,-7.25);
\foreach \y in {0,1,3,4,5,7} 
\draw [dashed, ->] (3.9,-1.5*\y)--(5.1,-1.5*\y);
\draw [dashed, ->] (3.9,-1.5*3.5)--(5.1,-1.5*3.5);

\foreach \x in {1,...,5} 
\draw [thick, ->] (3*\x,-9.25)--(3*\x,-10.25);
\draw [thick, ->] (3*3.5,-9.25)--(3*3.5,-10.25);
\foreach \y in {2,...,6} 
\draw [thick, ->] (0.9,-1.5*\y)--(2.1,-1.5*\y);
\draw [thick, ->] (0.9,-1.5*3.5)--(2.1,-1.5*3.5);

\draw [->] (9.2,-5.25)--(10.3,-5.25); 
\draw [->] (10.5,-5.35)--(10.5,-5.9);

\foreach \x in {21}
\foreach \y in {10.5}
\draw [red] (-0.7+\x,0.25-\y)--(0.7+\x,0.25-\y)--(0.7+\x,-0.25-\y)--(-0.7+\x,-0.25-\y)--(-0.7+\x,0.25-\y);

\foreach  \x in {21}
\draw [blue](-0.7+\x,0.25-1.5)--(0.7+\x,0.25-1.5)--(0.7+\x,-0.25-9)--(-0.7+\x,-0.25-9)--(-0.7+\x,0.25-1.5);

\foreach \y in {10.5}
\draw [blue](-0.7+3,0.25-\y)--(0.7+18,0.25-\y)--(0.7+18,-0.25-\y)--(-0.7+3,-0.25-\y)--(-0.7+3,0.25-\y);

\draw [red] (-0.7+3,1-3)--(-0.7+3,-0.25-9)--(0.25+16.5,-0.25-9)--(0.25+10.5,-0.25-6)--(0.25+10.5,0.25-5.25)--(-0.7+9,0.25-5.25)--(-0.7+3,1-3);

\draw [foge] (-0.7+3,0.25-1.5)--(0.7+18,0.25-1.5)--(0.7+18,-0.25-9)--(-0.7+18,-0.25-9)--(-0.7+12,-0.25-6)--(-0.7+12,-0.25-4.5)--(-0.7+9,-0.25-4.5)--(-0.7+3,-0.25-1.5)--(-0.7+3,0.25-1.5);

\foreach \x in {0,...,7} 
\filldraw (3*\x,-5.25) circle(2pt);
\foreach \y in {0,1,2,3,3.5,4,5,6,7} 
\filldraw (10.5,-1.5*\y) circle(2pt);

\end{tikzpicture}
\end{center}
\end{fig}
We then consider the set of colors $\C=\{c_1,\ldots,c_n,c_{\overline{0}},c_{\overline{n}},\ldots,c_{\overline{1}},c_0\}$, and by setting $\ep'(c_u,c_v)=H(v\ot u)$ and $H(0\ot 0)=0$, we obtain the following energy matrix for $\ep'$:
\begin{equation}\label{eq:energyD}
\bordermatrix{
\text{}&c_{1}&\cdots&c_{n}&c_{\overline{0}}&c_{\overline{n}}&\cdots&c_{\overline{1}}&c_{0}
\cr c_{1} &2&\cdots&2&2&2&\cdots&2&1
\cr \vdots&\vdots&\ddots&\vdots&\vdots&\vdots&2^{\star}&\vdots&\vdots
\cr c_{n}&0&\cdots&2&2&2&\cdots&2&1
\cr c_{\overline{0}}&0&\cdots&0&0&2&\cdots&2&1
\cr c_{\overline{n}}&0&\cdots&0&0&2&\cdots&2&1
\cr \vdots&\vdots&0^{\star}&\vdots&\vdots&\vdots&\ddots&\vdots&\vdots
\cr c_{\overline{1}}&0&\cdots&0&0&0&\cdots&2&1
\cr c_{0}&1&\cdots&1&1&1&\cdots&1&0
}\,\cdot
\end{equation}
\subsubsection{Character for $\Lambda_0$}
Here we set the ground to be $\co = c_0=1$. We obtain the energy matrix in \eqref{eq:energyD} by considering the energy matrix for $\ep$
\[
\bordermatrix{
\text{}&c_{1}&\cdots&c_{n}&c_{\overline{0}}&c_{\overline{n}}&\cdots&c_{\overline{1}}&c_{0}
\cr c_{1} &1&\cdots&1&1&1&\cdots&1&1
\cr \vdots&\vdots&\ddots&\vdots&\vdots&\vdots&1^{\star}&\vdots&\vdots
\cr c_{n}&0&\cdots&1&1&1&\cdots&1&1
\cr c_{\overline{0}}&0&\cdots&0&0&1&\cdots&1&1
\cr c_{\overline{n}}&0&\cdots&0&0&1&\cdots&1&1
\cr \vdots&\vdots&0^{\star}&\vdots&\vdots&\vdots&\ddots&\vdots&\vdots
\cr c_{\overline{1}}&0&\cdots&0&0&0&\cdots&1&1
\cr c_{0}&0&\cdots&0&0&0&\cdots&0&0
}
\]
followed by the transformation 
\begin{equation}\label{eq:transfo2}
(q,c_{\overline{0}},c_1,c_{\overline{1}},\ldots,c_n,c_{\overline{n}})\mapsto(q^2,c_{\overline{0}}q^{-1},c_1q^{-1},c_{\overline{1}}q^{-1},\ldots,c_nq^{-1},c_{\overline{n}}q^{-1})\,\cdot
\end{equation}
By applying \Thm{theo:flatreg1} to the corresponding flat partitions with ground $c_0$ and energy $\ep$, we obtain the generation function
\[
\sum_{\pi\in \Flt } C(\pi)q^{|\pi|}=\sum_{\pi\in \Reg } C(\pi)q^{|\pi|} = \frac{(-c_1q,-c_{\overline{1}}q,\ldots,-c_nq,-c_{\overline{n}}q;q)_{\infty}}{(c_{\overline{0}}q;q)}\,\cdot
\]
In fact, by the definition of the energy $\ep$, one can view the partitions of $\Reg$ as the finite sub-sequences, ending with $0_{c_0}$, of the infinite sequence 
\[\cdots\od 3_{c_1}\od 2_{c_{\overline{1}}}\od\cdots\od 2_{c_1}\od 1_{c_{\overline{1}}}\od \cdots\od 1_{c_{\overline{n}}}\od 1_{c_{\overline{0}}}\od 1_{c_n}\od \cdots\od 1_{c_1}\od 0_{c_0}\cdot\]
with the parts $k_{c_{\overline{0}}}$ possibly repeated.
Using \eqref{eq:transfo2}, we then have that the flat partitions with ground $c_0$ and energy $\ep'$ are  generated by the function
\[
\frac{(-c_1q,-c_{\overline{1}}q,\ldots,-c_nq,-c_{\overline{n}}q;q^2)_{\infty}}{(c_{\overline{0}}q;q^2)}\,\cdot
\]
By \Thm{theo:formchar}, \eqref{eq:weightD} and the fact that $c_{\overline{0}}=1$ with the convention of \Thm{theo:formchar}, we finally obtain the formula for the character corresponding to $\Lambda_0$.
\subsubsection{Character for $\Lambda_n$}
Here we set the ground to be $\co = c_{\overline{0}}=1$. We obtain the energy matrix in \eqref{eq:energyD} by considering the energy matrix of $\ep$
\[
\bordermatrix{
\text{}&c_1&\cdots&c_n&c_{\overline{0}}&c_{\overline{n}}&\cdots&c_{\overline{1}}&c_0
\cr c_1 &1&\cdots&1&1&0&\cdots&0&0
\cr \vdots&\vdots&\ddots&\vdots&\vdots&\vdots&0^{\star}&\vdots&\vdots
\cr c_n&0&\cdots&1&1&0&\cdots&0&0
\cr c_{\overline{0}}&0&\cdots&0&0&0&\cdots&0&0
\cr c_{\overline{n}}&1&\cdots&1&1&1&\cdots&1&1
\cr \vdots&\vdots&1^{\star}&\vdots&\vdots&\vdots&\ddots&\vdots&\vdots
\cr c_{\overline{1}}&1&\cdots&1&1&0&\cdots&1&1
\cr c_0&1&\cdots&1&1&0&\cdots&0&0
}
\quad\equiv\quad
\bordermatrix{
\text{}&c_{\overline{n}}&\cdots&c_{\overline{1}}&c_{0}&c_{1}&\cdots&c_{n}&c_{\overline{0}}
\cr c_{\overline{n}} &1&\cdots&1&1&1&\cdots&1&1
\cr \vdots&\vdots&\ddots&\vdots&\vdots&\vdots&1^{\star}&\vdots&\vdots
\cr c_{\overline{1}}&0&\cdots&1&1&1&\cdots&1&1
\cr c_{0}&0&\cdots&0&0&1&\cdots&1&1
\cr c_{1}&0&\cdots&0&0&1&\cdots&1&1
\cr \vdots&\vdots&0^{\star}&\vdots&\vdots&\vdots&\ddots&\vdots&\vdots
\cr c_{n}&0&\cdots&0&0&0&\cdots&1&1
\cr c_{\overline{0}}&0&\cdots&0&0&0&\cdots&0&0
}\,,\]
followed by the transformation 
\begin{equation}\label{eq:transfo3}
(q,c_0,c_{\overline{1}},\ldots,c_{\overline{n}})\mapsto(q^2,c_0q^{-1},c_{\overline{1}}q^{-2},\ldots,c_{\overline{n}}q^{-2})\,\cdot
\end{equation}
Here the part  $k_{c_0}$ for $\ep$ is transformed into $(2k-1)_{c_0}$ for $\ep'$, and the part $k_{c_{\overline{i}}}$ into $(2k-2)_{c_{\overline{i}}}$. Since $c_{\overline{0}}$ and $c_{i}$ are not modified, the part $k_{c}$ then becomes $(2k)_c$ for any $c \in \{c_{\overline{0}},c_i : i \in \{1,\ldots,n\}\}$.
\m
Applying \Thm{theo:flatreg1} to the flat partitions with ground $c_{\overline{0}}$ and energy $\ep$, and we obtain the generation function
\[
\sum_{\pi\in \Flt } C(\pi)q^{|\pi|}=\sum_{\pi\in \Reg } C(\pi)q^{|\pi|} = \frac{(-c_1q,-c_{\overline{1}}q,\ldots,-c_nq,-c_{\overline{n}}q;q)_{\infty}}{(c_{0}q;q)}\,\cdot
\]
In fact, by the definition of the energy $\ep$, one can view the partitions of $\Reg$ as the finite sub-sequences, ending with $0_{c_{\overline{0}}}$, of the infinite sequence 
\[\cdots\od 3_{c_{\overline{n}}}\od 2_{c_{n}}\od\cdots\od 2_{c_{\overline{n}}}\od 1_{c_n}\od \cdots\od 1_{c_1}\od 1_{c_0}\od 1_{c_{\overline{1}}}\od \cdots\od 1_{c_{\overline{n}}}\od 0_{c_{\overline{0}}}\cdot\]
with the parts $k_{c_{0}}$ possibly repeated.
Using \eqref{eq:transfo3}, we then have that the flat partitions with ground $c_{\overline{0}}$ and energy $\ep'$ are  generated by the function
\[
\frac{(-c_1q^2,-c_{\overline{1}},\ldots,-c_nq^2,-c_{\overline{n}};q^2)_{\infty}}{(c_{0}q;q^2)}\,\cdot
\]
By \Thm{theo:formchar}, \eqref{eq:weightD} and the fact that $c_0=1$ with the convention of \Thm{theo:formchar}, we obtain the formula for the character corresponding to $\Lambda_n$.
\subsection{Case of affine type $B_{n}^{(1)}(n\geq 3)$}
The crystal graph of the vector representation $\B$ of  $B_{n}^{(1)}(n\geq 3)$ is the following,
\begin{fig}
\[\]
\begin{center}
\begin{tikzpicture}[scale=0.6, every node/.style={scale=0.6}]
\draw (-4,0) node{$\B$ :};
\draw (-4,-0.75) node{$\p_{\Ll_n}=(\cdots\,\overline{0}\,\overline{0}\,\overline{0}\,\overline{0})$};
\draw (-4,-1.5) node{$\p_{\Ll_1}=(\cdots\,1\,\overline{1}\,1\,\overline{1}\,1)$};
\draw (-4,-2.25) node{$\p_{\Ll_0}=(\cdots\,\overline{1}\,1\,\overline{1}\,1\,\overline{1})$};
\draw (8.25,-0.75) node {$\overline{0}$};
\draw (0,0) node {$1$};
\draw (1.5,0) node {$2$};
\draw (5,0) node {$n-1$};
\draw (6.75,0) node {$n$};
\draw (3.15,0) node {$\cdots$};
\draw (0,-1.5) node {$\overline{1}$};
\draw (1.5,-1.5) node {$\overline{2}$};
\draw (5,-1.5) node {$\overline{n-1}$};
\draw (6.75,-1.5) node {$\overline{n}$};
\draw (3.15,-1.5) node {$\cdots$};

\draw (-0.2+8.25,0.2-0.75)--(0.2+8.25,0.2-0.75)--(0.2+8.25,-0.2-0.75)--(-0.2+8.25,-0.2-0.75)--(-0.2+8.25,0.2-0.75);
\foreach \x in {0,1,4.5} 
\foreach \y in {0,1}
\draw (-0.2+1.5*\x,0.2-1.5*\y)--(0.2+1.5*\x,0.2-1.5*\y)--(0.2+1.5*\x,-0.2-1.5*\y)--(-0.2+1.5*\x,-0.2-1.5*\y)--(-0.2+1.5*\x,0.2-1.5*\y);
\foreach \x in {0} 
\foreach \y  in {0,1}
\draw (-0.45+5+2*\x,0.2-1.5*\y)--(0.45+5+2*\x,0.2-1.5*\y)--(0.45+5+2*\x,-0.2-1.5*\y)--(-0.45+5+2*\x,-0.2-1.5*\y)--(-0.45+5+2*\x,0.2-1.5*\y);

\draw [thick, <-] (0.05,-0.25)--(1.45,-1.25);
\draw [thick, <-] (1.45,-0.25)--(0.05,-1.25);
\draw (0.3,-0.9) node {\footnotesize{$0$}};
\draw (1.2,-0.9) node {\footnotesize{$0$}};

\draw [thick, ->] (0.3,0)--(1.2,0);
\draw [thick, <-] (0.3,-1.5)--(1.2,-1.5);
\draw (0.75,0.15) node {\footnotesize{$1$}};
\draw (0.75,-1.67) node {\footnotesize{$1$}};
\draw [thick, ->] (1.8,0)--(2.7,0);
\draw [thick, <-] (1.8,-1.5)--(2.7,-1.5);
\draw (2.25,0.15) node {\footnotesize{$2$}};
\draw (2.25,-1.67) node {\footnotesize{$2$}};
\draw [thick, ->] (3.6,0)--(4.5,0);
\draw [thick, <-] (3.6,-1.5)--(4.5,-1.5);
\draw (4.05,0.15) node {\footnotesize{$n-2$}};
\draw (4.05,-1.67) node {\footnotesize{$n-2$}};
\draw [thick, ->] (5.6,0)--(6.4,0);
\draw [thick, <-] (5.6,-1.5)--(6.4,-1.5);
\draw (6,0.15) node {\footnotesize{$n-1$}};
\draw (6,-1.67) node {\footnotesize{$n-1$}};
\draw [thick, ->] (7.05,-0.05)--(8.2,-0.5);
\draw [thick, <-] (7.05,-1.45)--(8.2,-1);
\draw (7.6,-0.1) node {\footnotesize{$n$}};
\draw (7.6,-1.4) node {\footnotesize{$n$}};

\end{tikzpicture}
\end{center}
\end{fig}
with $\wt (\overline{0})=0$ and for all $u\in \{1,\ldots,n\}$,
\begin{equation}\label{eq:weightB}
-\wt (\overline{u}) = \wt u = \sum_{i=u}^{n} \alpha_i\,\cdot
\end{equation}
Here, we have $\delta = \alpha_0+\alpha_1+2\sum_{i=2}^{n} \alpha_i$.
We thus obtain the following crystal graph for $\B\ot \B$ 
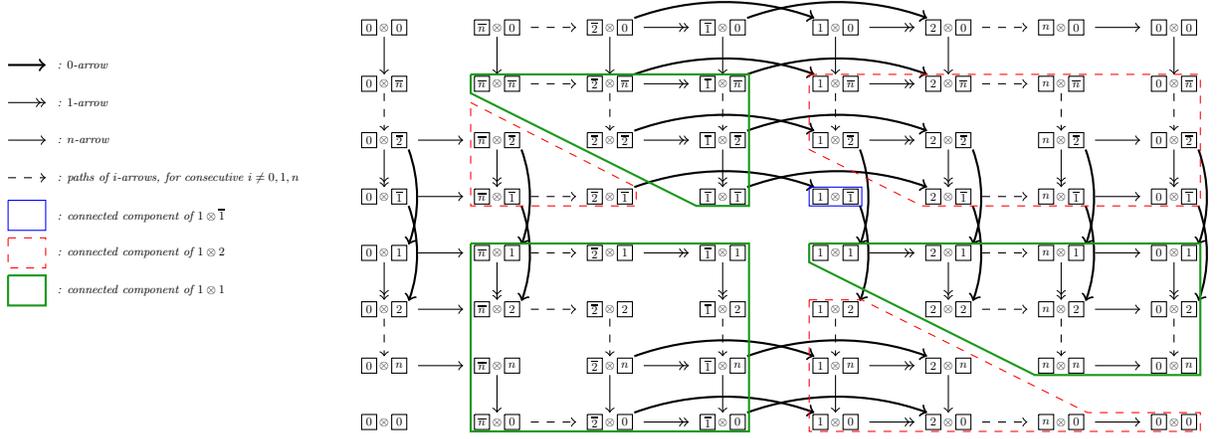
\begin{fig}
\[\]
\begin{center}
\begin{tikzpicture}[scale=0.5, every node/.style={scale=0.5}]
\draw [thick,->] (-10,-1)--(-9,-1);
\draw (-8.8,-1) node[right] {: $0$-arrow};
\draw [->>] (-10,-2)--(-9,-2);
\draw (-8.8,-2) node[right] {: $1$-arrow};
\draw [->] (-10,-3)--(-9,-3);
\draw (-8.8,-3) node[right] {: $n$-arrow};
\draw [dashed,->] (-10,-4)--(-9,-4);
\draw (-8.8,-4) node[right] {: paths of $i$-arrows, for consecutive $i\neq 0,1,n$};
\draw [blue] (-10,-4.6)--(-9,-4.6)--(-9,-5.4)--(-10,-5.4)--(-10,-4.6);
\draw (-8.8,-5) node[right] {: connected component of $1\ot \overline{1}$};
\draw [dashed,red] (-10,-5.6)--(-9,-5.6)--(-9,-6.4)--(-10,-6.4)--(-10,-5.6);
\draw (-8.8,-6) node[right] {: connected component of $1\ot 2$};
\draw [thick,foge] (-10,-6.6)--(-9,-6.6)--(-9,-7.4)--(-10,-7.4)--(-10,-6.6);
\draw (-8.8,-7) node[right] {: connected component of $1\ot 1$};
\foreach \x in {0,...,7} 
\foreach \y in {0,...,7}
\draw (-0.2-0.4+3*\x,0.2-1.5*\y)--(0.2-0.4+3*\x,0.2-1.5*\y)--(0.2-0.4+3*\x,-0.2-1.5*\y)--(-0.2-0.4+3*\x,-0.2-1.5*\y)--(-0.2-0.4+3*\x,0.2-1.5*\y);
\foreach \x in {0,...,7} 
\foreach \y in {0,...,7}
\draw (-0.2+0.4+3*\x,0.2-1.5*\y)--(0.2+0.4+3*\x,0.2-1.5*\y)--(0.2+0.4+3*\x,-0.2-1.5*\y)--(-0.2+0.4+3*\x,-0.2-1.5*\y)--(-0.2+0.4+3*\x,0.2-1.5*\y);

\foreach \x in {0,...,7} 
\foreach \y in {0,...,7}
\draw (3*\x,-1.5*\y) node {$\ot$};

\foreach \x in {0,...,7} 
\draw (0.4+3*\x,0) node {$0$};
\foreach \y in {0,...,7} 
\draw (-0.4,-1.5*\y) node {$0$};
\foreach \x in {0,...,7} 
\draw (0.4+3*\x,-10.5) node {$0$};
\foreach \y in {0,...,7} 
\draw (-0.4+21,-1.5*\y) node {$0$};

\foreach \x in {0,...,7}
\draw (0.4+3*\x,-1.5) node {$\overline{n}$};
\foreach \y in {0,...,7} 
\draw (-0.4+3,-1.5*\y) node {$\overline{n}$};

\foreach \x in {0,...,7}
\draw (0.4+3*\x,-3) node {$\overline{2}$};
\foreach \y in {0,...,7} 
\draw (-0.4+6,-1.5*\y) node {$\overline{2}$};

\foreach \x in {0,...,7}
\draw (0.4+3*\x,-4.5) node {$\overline{1}$};
\foreach \y in {0,...,7} 
\draw (-0.4+9,-1.5*\y) node {$\overline{1}$};

\foreach \x in {0,...,7}
\draw (0.4+3*\x,-6) node {$1$};
\foreach \y in {0,...,7} 
\draw (-0.4+12,-1.5*\y) node {$1$};

\foreach \x in {0,...,7}
\draw (0.4+3*\x,-7.5) node {$2$};
\foreach \y in {0,...,7} 
\draw (-0.4+15,-1.5*\y) node {$2$};

\foreach \x in {0,...,7}
\draw (0.4+3*\x,-9) node {$n$};
\foreach \y in {0,...,7} 
\draw (-0.4+18,-1.5*\y) node {$n$};

\foreach \x in {0,...,5,7} 
\draw [->] (3*\x,-0.25)--(3*\x,-1.25);
\foreach \x in {1,...,5} 
\draw [->] (3*\x,-9.25)--(3*\x,-10.25);

\foreach \x in {0,2,3,4,6,7} 
\draw [dashed,->] (3*\x,-1.75)--(3*\x,-2.75);
\foreach \x in {0,2,3,4,6,7} 
\draw [dashed,->] (3*\x,-7.75)--(3*\x,-8.75);

\foreach \x in {2,4}
\foreach \y in {0,1,2,4,6,7}
\draw [->>] (0.9+3*\x,-1.5*\y)--(2.1+3*\x,-1.5*\y);

\foreach \x in {2,3}
\foreach \y in {0,1,2,3,6,7}
\draw [thick,->] (0.65+3*\x,0.25-1.5*\y) arc (110:70:7);

\foreach \y in {2,4}
\foreach \x in {0,1,3,5,6,7}
\draw [->>] (3*\x,-0.25-1.5*\y)--(3*\x,-1.25-1.5*\y);

\foreach \y in {2,3}
\foreach \x in {0,1,4,5,6,7}
\draw [thick,->] (0.65+3*\x,-0.25-1.5*\y) arc (20:-20:3.7);

\foreach \y in {2,...,6} 
\draw [->] (0.9,-1.5*\y)--(2.1,-1.5*\y);
\foreach \y in {0,2,3,4,5,6,7} 
\draw [->] (18.9,-1.5*\y)--(20.1,-1.5*\y);

\foreach \y in {0,1,3,4,5,7} 
\draw [dashed,->] (3.9,-1.5*\y)--(5.1,-1.5*\y);
\foreach \y in {0,1,3,4,5,7} 
\draw [dashed,->] (15.9,-1.5*\y)--(17.1,-1.5*\y);

\draw [blue] (-0.7+12,0.25-4.5)--(0.7+12,0.25-4.5)--(0.7+12,-0.25-4.5)--(-0.7+12,-0.25-4.5)--(-0.7+12,0.25-4.5);

\draw [dashed,red] (-0.7+12,0.25-1.5)--(0.7+21,0.25-1.5)--(0.7+21,-0.25-4.5)--(-0.7+15,-0.25-4.5)--(-0.7+12,-0.25-3)--(-0.7+12,0.25-1.5);
\draw [dashed,red] (-0.7+3,1-3)--(0.7+6,0.25-4.5)--(0.7+6,-0.25-4.5)--(-0.7+3,-0.25-4.5)--(-0.7+3,1-3);
\draw [dashed,red] (-0.7+12,0.25-7.5)--(0.7+12,0.25-7.5)--(0.7+18,0.25-10.5)--(0.7+21,0.25-10.5)--(0.7+21,-0.25-10.5)--(-0.7+12,-0.25-10.5)--(-0.7+12,0.25-7.5);

\draw [thick,foge] (-0.7+3,0.25-6)--(0.7+9,0.25-6)--(0.7+9,-0.25-10.5)--(-0.7+3,-0.25-10.5)--(-0.7+3,0.25-6);
\draw [thick,foge] (-0.7+3,0.25-1.5)--(0.7+9,0.25-1.5)--(0.7+9,-0.25-4.5)--(-0.7+9,-0.25-4.5)--(-0.7+3,-0.25-1.5)--(-0.7+3,0.25-1.5);
\draw [thick,foge] (-0.7+12,0.25-6)--(0.7+21,0.25-6)--(0.7+21,-0.25-9)--(-0.7+18,-0.25-9)--(-0.7+12,-0.25-6)--(-0.7+12,0.25-6);
\end{tikzpicture}
\end{center}
\end{fig}
Here, the only suitable ground to apply \Thm{theo:formchar} is $c_{\overline{0}}$.
We then consider $\C=\{c_1,\ldots,c_n,c_{\overline{n}},\ldots,c_{\overline{1}},c_{\overline{0}}\}$, and by setting $\ep'(c_u,c_v)=H(v\ot u)$ and $H(\overline{0}\ot \overline{0})=0$, we obtain the following energy matrix for $\ep'$:
\[
\bordermatrix{
\text{}&c_{\overline{n}}&\cdots&c_{\overline{2}}&c_{\overline{1}}&c_{1}&c_2&\cdots&c_{n}&c_{\overline{0}}
\cr c_{\overline{n}} &1&\cdots&1&1&0&0&\cdots&0&0
\cr \vdots&0&\ddots&\vdots&\vdots&\vdots&\vdots&0^{\star}&\vdots&\vdots
\cr c_{\overline{2}}&0&\cdots&1&1&0&0&\cdots&0&0
\cr c_{\overline{1}}&0&\cdots&0&1&-1&0&\cdots&0&0
\cr c_1&1&\cdots&1&1&1&1&\cdots&1&1
\cr c_2&1&\cdots&1&1&0&1&\cdots&1&1
\cr \vdots&\vdots&1^{\star}&\vdots&\vdots&\vdots&\vdots&\ddots&\vdots&\vdots
\cr c_n&1&\cdots&1&1&0&0&\cdots&1&1
\cr c_{\overline{0}}&1&\cdots&1&1&0&0&\cdots&0&0
}\,\cdot
\]
This energy matrix can be obtain by taking the energy matrix of $\ep$ defined by
\[
\bordermatrix{
\text{}&c_{\overline{n}}&\cdots&c_{\overline{2}}&c_{\overline{1}}&c_{1}&c_2&\cdots&c_{n}&c_{\overline{0}}
\cr c_{\overline{n}} &1&\cdots&1&1&1&1&\cdots&1&1
\cr \vdots&0&\ddots&\vdots&\vdots&\vdots&\vdots&1^{\star}&\vdots&\vdots
\cr c_{\overline{2}}&0&\cdots&1&1&1&1&\cdots&1&1
\cr c_{\overline{1}}&0&\cdots&0&1&0&1&\cdots&1&1
\cr c_1&0&\cdots&0&0&1&1&\cdots&1&1
\cr c_2&0&\cdots&0&0&0&1&\cdots&1&1
\cr \vdots&\vdots&0^{\star}&\vdots&\vdots&\vdots&\vdots&\ddots&\vdots&\vdots
\cr c_n&0&\cdots&0&0&0&0&\cdots&1&1
\cr c_{\overline{0}}&0&\cdots&0&0&0&0&\cdots&0&0
}\,,
\]
followed by the transformation 
\begin{equation}\label{eq:transfo4}
(q,c_{\overline{1}},\ldots,c_{\overline{n}})\mapsto(q,c_{\overline{1}}q^{-1},\ldots,c_{\overline{n}}q^{-1})\,\cdot
\end{equation}
Here the part $k_{c_{\overline{i}}}$ for $\ep$  is transformed into $(k-1)_{c_{\overline{i}}}$ for $\ep'$. The other parts $k_{c}$ remain unchanged.
By setting the ground $\co = c_{\overline{0}}=1$, we can apply \Thm{theo:flatreg1} to the flat partitions generated by $\ep$, and we obtain the generation function
\[
\sum_{\pi\in \Flt } C(\pi)q^{|\pi|}=\sum_{\pi\in \Reg } C(\pi)q^{|\pi|} = \frac{(-c_1q,-c_{\overline{1}}q,\ldots,-c_nq,-c_{\overline{n}}q;q)_{\infty}}{(c_1c_{\overline{1}}q^2;q^2)_\infty}\,\cdot
\]
In fact, by the definition of the energy $\ep$, one can view the partitions of $\Reg$ as the finite sub-sequences, ending with $0_{c_{\overline{0}}}$, of the infinite sequence 
\[\cdots\od 3_{c_{\overline{n}}}\od 2_{c_{n}}\od\cdots\od 2_{c_{\overline{n}}}\od 1_{c_n}\od \cdots \od 1_{c_2}\od 1_{c_{\overline{1}}}\od 1_{c_1}\od 1_{c_{\overline{1}}}\od 1_{c_{\overline{2}}}\od \cdots\od 1_{c_{\overline{n}}}\od 0_{c_{\overline{0}}}\,,\]
with the additional condition that we have possibly alternating sub-sequences of the form
\[\cdots \od k_{c_1}\od k_{c_{\overline{1}}}\od k_{c_1}\od k_{c_{\overline{1}}}\od\cdots \cdot\]
By reasoning on the parity of the length and the first element, the generating function of such alternating sequences for a fixed size $k$, possibly empty or reduced to one element,  is equal to 
\[\frac{(1+c_1q^k)(1+c_{\overline{1}}q^k)}{1-c_1c_{\overline{1}}q^{2k}}\,\cdot\] 
Using \eqref{eq:transfo4}, we then have that the flat partitions with ground $\co$ and energy $\ep'$ are  generated by the function
\[
\frac{(-c_1q,-c_{\overline{1}},\ldots,-c_nq,-c_{\overline{n}};q)_{\infty}}{(c_1c_{\overline{1}}q;q^2)_\infty}\,\cdot
\]
Using \Thm{theo:formchar} and \eqref{eq:weightB}, we obtain the formula for the character corresponding to $\Lambda_n$ in \Thm{theo:char3}.
\section{Bijective proof of \Thm{theo:flatreg1}}\label{sect:deg1}
We build in this section a bijection $\Omega_1$ between the set $\Flt$ and $\Reg$ of \Thm{theo:flatreg1}. In the following, we illustrate $\Omega_1$ with the set of colors $\C=\{a,b,c\}$, the ground $c$, and the energy $\ep$ defined by the energy matrix
\[M_\ep =
\bordermatrix{
\text{}&a&b&c
\cr a&1&0&1
\cr b&0&0&1
\cr c&0&0&0
}\,\cdot
\]
\subsection{The setup}
Recall that $\delta_g$ be the common value of the $\ep(\co,c)$ for $c\neq \co$ given by \eqref{eq:pos0}. Note that for any $c\neq \co$, for any $k,l \in \Z$
 \begin{align}\label{eq:order}
  k_{c}\not\od l_{\co} &\Longleftrightarrow k-l\leq \ep(c,\co)-1 \nonumber\\
  &\Longleftrightarrow l-k\geq 1-\ep(c,\co)\nonumber\\
  &\Longleftrightarrow l-k\geq \ep(\co,c)\nonumber\\
  &\Longleftrightarrow l_{\co}\od k_{c}\\
 \end{align}
so that \textit{the parts with color $\co$ can be always related in terms of $\od$ with the parts with color different from $\co$.} 
 \bi 
One can see the classical integer partitions as the non-increasing sequences of non-negative integers, with all but a finite number of parts equal to $0$. 
 \m Let us recall the conjugacy on classical partitions.
The partitions $\nu = (\nu_i)_{i=0}^{\infty}$ and $\nu'=(\nu'_i)_{i=0}^{\infty}$ are conjugate if and only if  their part sizes satisfy 
 \begin{equation}\label{eq:conj}
 \nu_i = |\{\nu'_j\geq i+1\}| 
\end{equation}
The conjugacy is an involution, and we then have $\nu'_i = |\{\nu_j\geq i+1\}|$.
\subsubsection{The set $\Reg$}
Identify a partition $\pi= (\pi_0,\ldots,\pi_{s-1},0_{\co})$ of $\Reg$  as the unique pair of partitions 
 \[(\mu,\nu)=[(\mu_0,\ldots,\mu_{s-1},0_{c_0}),(\nu_0,\ldots,\nu_{s-1})]\,,\]
 such that $C(\pi)=C(\mu)=c_{0}\cdots c_{s-1}\co$, and for all $k\in \ssss$, we have $c_k\neq \co$,
 $$\mu_k = \left(\sum_{l=k}^{s-1}\ep(c_k,c_{k+1})\right)_{c_k}\quad\text{and}\quad \nu_k = |\pi_k|-|\mu_k|\,\cdot$$
 The partition $\mu$ is then the unique element in $\Flt\cap \Reg$ satisfying $C(\pi)=C(\mu)=c_{0}\cdots c_{s-1}\co$, and $\nu$ is the \textit{residual} partition with $s$ parts, possibly ending with some parts equal to $0$. The partition $\nu$ then corresponds to a unique classical partition, with at most $s$ parts. 
 \begin{ex}\label{ex:bij1}
 The partition $$\pi=(10_a,8_a,8_b,7_b,5_a,4_a,3_a,2_b,1_a,1_b,1_b,0_c)$$
 is identify with the pair $(\mu,\nu)$ with 
 $$\mu = (4_a,3_a,3_b,3_b,3_a,2_a,1_a,1_b,1_a,1_b,1_b,0_c)$$
 and 
 $$\nu=(6,5,5,4,2,2,2,1,0,0,0)\,\cdot$$
 \end{ex}
 \bi 
We now fix $C=c_{0}\cdots c_{s-1}$. The partition $\mu$ in the pair then becomes fixed. By considering the set of regular partitions in $\Reg$ with a color sequence in $C\co$, we have the bijection
  \begin{equation}
  \Reg(C)=\{\pi\in \Reg: C(\pi)=C\co\}\approx \{\mu\}\times\{(\nu_0,\cdots,\nu_{s-1})\in \Z_{\geq 0}: \nu_0\geq \cdots \geq \nu_{s-1}\}\,\cdot
 \end{equation}
The set $\Reg(C)$ is then isomorphic to the set of classical partitions with at most $s$ positive parts. 
\m
We now consider the set of the \textit{descents}
 \begin{equation}
  D=\{k:\sss: \ep(c_{k-1},c_k)=0\}=\{k_0<\cdots<k_{|D|-1}\}\quad \text{and}\qquad\overline{D} = \ssss  \setminus D\,\,\cdot
 \end{equation}
 Note that, since $\ep(c_{s-1},\co)=1-\delta_g$, we recursively have for all $k\in \ssss$ that
 \begin{equation}\label{eq:mu}
  |\mu_k|= \sum_{l=k}^{s-1}\ep(c_l,c_{l+1}) = 1-\delta_g+|\{k+1,\cdots,s-1\}\cap \overline{D}|\leq s-k -\delta_g\,\cdot
 \end{equation}
 With \Expl{ex:bij1}, $C=aabbaaababb$, $s=11$, $D=\{2,3,4,7,8,9,10\}$ and $\overline{D}=\{0,1,5,6\}$.
 \bi 
 For a fixed non-negative $n$, we construct the bijection $\Omega$ in such a way that the partitions $\pi$ in $\Reg$ satisfying $(|\pi|,C(\pi))=(n,C\co)$ map to the 
 partitions $\pi$ in $\Flt$ which satisfy $(n,C)=(|\pi|,C(\pi)_{|\co=1})$. The latter means that the sequence of colors different from $\co$ 
 is equal to $C$.
\subsubsection{The set $\Flt$}
We now consider the set $\Flt(C)$ of flat partitions $\pi$ in $\Flt$ such that $C(\pi)_{|c_0=1}=C$. For such a partition $\pi$, there then exists a unique set $S=\{u_0<\cdots<u_{s-1}\}\subset \Z_{\geq 0}$ such that 
 \begin{equation}
 \begin{array}{ccl}
  \pi &=& (\pi_{0},\cdots,\pi_{u_{s-1}},0_{c_0})\,,\\
  c(\pi_{u_k})&=&c_{k} \quad \forall \,\,k\in \{0,\ldots,s-1\}\,,\\
  c_{\pi_k}&=&\co \quad \forall \,\,k\in\{0,\ldots,u_{s-1}\}\setminus S \,\cdot
 \end{array}
 \end{equation}
 In fact, we cannot have $c(\pi_{u_{s-1}})=\co$, otherwise $\pi_{u_{s-1}} = \ep(\co,\co) = 0$, so that $\pi_{u_{s-1}} = 0_{\co}$, and this contradicts the definition of grounded partitions. Let us set 
\begin{align}
 s'&= u_{s-1}+1-s\nonumber\\ W &= \{0\leq v <|D|: u_{k_v}-u_{k_v-1}>1\}= \{v_0<\cdots<v_{|W|-1}\}\,,\label{eq:ww}\\ 
 D_{W} &= \{k_v: v\in W\}\,,\nonumber\\
 D_{\overline{W}} &= D\setminus D_W\,\cdot\nonumber
\end{align}
If we have some parts with color $\co$ between $u_{k}$ and $u_{k+1}$, 
 their sizes' differences gives
 \begin{align*}
  \ep(c_k,\co)+\underbrace{0+\cdots+0}_{\sharp\text{parts inserted}-1}+\ep(\co,c_{k+1}) &= \ep(c_k,\co)+\ep(\co,c_{k+1})\\\
  &=1-\delta_g+\delta_g\\
  &=1
 \end{align*}
 and it differs from $\ep(c_k,c_{k+1})$ if and only if $k+1\in D$.
 We then obtain that 
 \[
 |\pi_{u_k}| = |\mu_k| + |\{k+1,\cdots,s-1\}\cap D_W|\,,
 \]
 so that by \eqref{eq:mu}. Since $|\pi_{u_{s-1}}|=1-\delta_g$, we obtain recursively that for all $k\in\{0,\ldots,s-2\}$,
 \begin{equation}\label{eq:piu}
  |\pi_{u_k}|=1-\delta_g+|\{k+1,\ldots,s-1\}\cap (\overline{D}\sqcup D_{W})|\,\cdot                       
 \end{equation}
 Note that by \eqref{eq:piu}, for all $u_{k-1}< u <u_k$, we necessarily have that $k\in \overline D \sqcup D_W$, and then 
 \[
 |\pi_{u}|  = \delta_g + |\pi_{u_k}| = |\{k,\ldots,s-1\}\cap (\overline{D}\sqcup D_{W})|\,\cdot
 \]
\bi
 \subsection{The map $\Omega$ from $\Flt(C)$ to $\Reg(C)$.}
 For any partition $\pi\in  \Flt(C)$ described above, we associate the conjugate $\nu$ of the classical partition $\nu'$ whose parts are:
 \begin{enumerate}
  \item for $k\notin D$, the $u_k-u_{k-1}-1$ parts between $u_{k-1}$ and $u_k$ with size 
  \[|\pi_u| = |\{k,\ldots,s-1\}\cap (\overline{D}\sqcup D_{W})|\,,\]
  which the convention $u_{-1}=-1$.
  \item for $k\in D_W$, we take $u_k-u_{k-1}-2$ parts between $u_{k-1}$ and $u_k$ with size 
  \[|\pi_u| =|\{k,\ldots,s-1\}\cap (\overline{D}\sqcup D_{W})|\,,\]
 and one part (called the \textit{weighted} part) with size
 \begin{equation}\label{eq:reci}
 |\pi_u|+k =|\{k,\ldots,s-1\}\cap (\overline{D}\sqcup D_{W})|+k \,\cdot
 \end{equation}
 \end{enumerate}
\bi
\begin{ex}
For example, we illustrate these transformations with $C=aabbaaababb$ and
\[\pi=(6_a,5_a,5_b,4_{c},4_{c},4_{c},4_b,4_a,3_{c},3_a,2_a,1_{c},1_{c},1_b,1_a,1_b,1_b,0_c)\,\cdot\]
Recall that $\mu = (4_a,3_a,3_b,3_b,3_a,2_a,1_a,1_b,1_a,1_b,1_b,0_c)$, $D=\{2,3,4,7,8,9,10\}$ and $\overline{D}=\{0,1,5,6\}$.
Here we have that 
\[
\begin{array}{|c|ccccccccccc|}
 \hline
 k&0&1&2&3&4&5&6&7&8&9&10\\
 \hline
 u_k&0&1&2&6&7&9&10&13&14&15&16\\
 \hline
 \end{array}
\]
and then 
$D_W=\{3,7\}$. We thus obtain that $\nu'$ is the classical partition with parts $3,4,4,7$ and $1,8$. We thus have $\nu'=(8,7,4,4,3,1)$
and the conjugacy then gives the following partition with $11$ parts
\[\nu = (6,5,5,4,2,2,2,1,0,0,0)\,\cdot\] 
By adding the parts of $\nu$ to the corresponding parts of $\mu$, we finally obtain
\[
\Omega(6_a,5_a,5_b,4_{c},4_{c},4_{c},4_b,4_a,3_{c},3_a,2_a,1_{c},1_{c},1_b,1_a,1_b,1_b,0_c)
=(10_a,8_a,8_b,7_b,5_a,4_a,3_a,2_b,1_a,1_b,1_b,0_c)\,\cdot
\]
\end{ex}
\bi 
We first note that the size of the partition is conserved by these transformations, since 
\begin{align*}
 \sum_{k=0}^{s-1}|\{k+1,\cdots,s-1\}\cap D_W| &= \sharp\{(k,l): l\in \{k+1,\cdots,s-1\}: l\in D_W\}\\
 &= \sum_{l\in D_W} \sharp\{0\leq k<l\}\\
 &= \sum_{l\in D_W}  l
\end{align*}
and then 
\begin{align*}
 \sum_{u=0}^{u_{s-1}} |\pi_u| &= \sum_{k=0}^{s-1}|\pi_{u_k}| + \sum_{\substack{k\\ 1<u_{k}-u_{k-1}}}(u_{k}-u_{k-1}-1)|\pi_{u_k-1}| \\
 &= \sum_{u\notin S} |\pi_u|+ \sum_{k=0}^{s-1} |\mu_k| + \sum_{l\in D_W}  l \\
 & = |\mu| + \sum_{l\in D_W}  l+|\pi_{u_l-1}| + (u_{l}-u_{l-1}-2)\pi_{u_l-1} + \sum_{l\notin D}(u_{l}-u_{l-1}-1)|\pi_{u_l-1}|\,\cdot
\end{align*}
\bi
The \textit{unweighted} parts are those which are not weighted. We then remark that for all $k\in \ssss$, 
\begin{align*}
 |\{k,\ldots,s-1\}\cap (\overline{D}\sqcup D_{W})|+k &= |\overline{D}\sqcup D_{W}|+ |\{0,\ldots,k-1\}\cap D_{\overline{W}}| \\
 &= |\pi_{u_0}|+\delta_g + |\{0,\ldots,k-1\}\cap D_{\overline{W}}|
\end{align*}
so that the weighted parts have all sizes greater than or equal to the sizes of the unweighted parts. We also notice that unweighted parts coming from different $k$ are distinct, 
since the sizes' difference gives
\begin{align*}
|\{k,\ldots,s-1\}\cap (\overline{D}\sqcup D_{W})|-|\{k+1,\ldots,s-1\}\cap (\overline{D}\sqcup D_{W})| = \chi(k\in \overline{D}\sqcup D_{W})
\end{align*}
and this is exactly the condition required to \textit{insert} a part in $\nu'$. Also when we take two consecutive weighted parts in $k_{v_j}<k_{v_{j+1}}\in D_W$, we obtain the difference of size
\begin{align*}
k_{v_{i}}-k_{v_{i+1}}+|\{k_{v_i},\ldots,k_{v_{i+1}}-1\}\cap (\overline{D}\sqcup D_{W})| = -|\{k_{v_i},\ldots,k_{v_{i+1}}-1\}\cap D_{\overline{W}}|
\end{align*}
so that the weighted parts appear in a non-decreasing order according to the indices $i$ in $\{0,|W|-1\}$. We then obtain 
$\nu' = (\nu'_{0},\cdots,\nu'_{s'-1})$, where
we have for all $i\in \{0,\ldots,|W|-1\}$ 
\begin{align*}
 \nu'_{|W|-1-i} &= |\{k_{v_i},s-1\}\cap (\overline{D}\sqcup D_{W})|+k_{v_i}\\
  &= s- |D_{\overline{W}}\cap\{k_{v_i},\ldots,s-1\}|\\
  &= s- |\{v_i\leq p<|D|: p\notin W\}|\quad\text{by \eqref{eq:ww}}\\
  &= s+|W|-|D|+ v_i-i\\
  &\leq s\,,
\end{align*}
and the rest of the parts consists of $u_k-u_{k-1}-1-\chi(k\in D_W)$ parts 
for $k\in \overline{D}\sqcup D_{W}$, each of them with size 
\[
 |\{k,s-1\}\cap (\overline{D}\sqcup D_{W})|\geq 1\,\cdot
\]
Note that $\nu'$ viewed as a classical partition has $s'$ parts with 
size at most equal to $s$, and by \eqref{eq:conj},
the partition $\nu$
then has at most $s$ positive parts and satisfies $\nu_0 = s'$. Our map from $\Flt(C)$ to $\Reg(C)$ is then well-defined.
\bi 
We conclude by observing the following equality: for all $i\in \{0,\ldots,|W|-1\}$ we have 
\begin{align}
 \nu'_{|W|-1-i}-|W|+i&= s-|D|+ v_i\nonumber\\
 &=|\{0,\ldots,k_{v_i}-1\}|+ |\{k_{v_i},s-1\}\cap \overline{D}|\nonumber\\
 &= \delta_g+\mu_{k_{v_i}}+ k_{v_i}\label{eq:retour1}\,,
 \end{align}
and for all $u\in \{|W|,\ldots,s'-1\}$, we have 
\begin{equation}
\nu'_{u}-u-1\leq \nu'_{|W|}-|W|-1<\delta_g+\mu_0\,\cdot\label{eq:retour2}
\end{equation}
\subsection{The map $\Omega^{-1}$ from $\Reg(C)$  to  $\Flt(C)$}
Let consider a partition $\pi$ in $\Reg(C)$, and the corresponding pair $(\mu,\nu)$.
The partition $\nu$ then corresponds to a classical partition with at most $s$ positive parts.
We first consider the partition $\nu'$ the conjugate of $\nu$ given by the relation \eqref{eq:conj}. The partitions $\nu'$ has then $\nu_0$ positive parts, whose sizes are at most equal to $s$.
Let us set $s'=\nu_0$ and write $\nu' = (\nu'_{0},\cdots,\nu'_{s'-1})$. 
We then apply the following transformations:
\begin{enumerate}
 \item For each $k\in \ssss$, we change the part $\mu_k$ into $\mu'_k$ with the relations 
 \begin{equation}
 \begin{cases}
 c(\mu'_k)= c(\mu_k) = c_{k}\\
  |\mu'_k|= |\mu_k| + |\{0\leq u<s': \delta_g+|\mu_k| +k \leq \nu'_{u}-u-1\}|
 \end{cases}\,\,\cdot
 \end{equation}
 \item For each $u\in \{0,\ldots,s'-1\}$, we change the part $\nu'_u$ into $\nu''_u$ with the relations 
 \begin{equation}
  \begin{cases}
 c(\nu''_u)= \co\\
  |\nu''_u|= \nu'_u - |\{0\leq k<s: \delta_g+|\mu_k| +k \leq \nu'_{u}-u-1\}|
 \end{cases}\,\,\cdot
 \end{equation}
\end{enumerate}
The final partition $\Omega^{-1}(\pi)$ is obtained by inserting the parts $\nu''_u$ into the sequence of parts $\mu'_k$ according $\od$, and adding the ground $0_{\co}$.
The partition $\Omega^{-1}(\pi)$ then has $s+s'$ parts different from $0_{\co}$ and by double counting, we have that $|\Omega^{-1}(\pi)|=|\mu|+|\nu| = |\pi|$.
\bi   
\begin{ex}
For example, we illustrate these transformations with $C=aabbaaababb$ and 
$$\pi=(10_a,8_a,8_b,7_b,5_a,4_a,3_a,2_b,1_a,1_b,1_b,0_c)\,,$$
corresponding to
$$\mu = (4_a,3_a,3_b,3_b,3_a,2_a,1_a,1_b,1_a,1_b,1_b,0_c)\,,$$
 and  
 $$\nu=(6,5,5,4,2,2,2,1,0,0,0)\,\cdot$$
 The conjugacy gives 
 $$\nu'=(8,7,4,4,3,1)$$
 Recall that we have $\delta_g=0$. Using the following tables
 \[
 \begin{array}{|c|ccccccccccc|}
 \hline
 k&0&1&2&3&4&5&6&7&8&9&10\\
 \hline
 \mu_k +k&4&4&5&6&7&7&7&8&9&10&11\\
 \hline
 \end{array}
\qquad\begin{array}{|c|cccccc|}
 \hline
 u&0&1&2&3&4&5\\
 \hline
 \nu'_u-u-1&7&5&1&0&-2&-5\\
 \hline
 \end{array}\,,\]
 we obtain that 
 \[\mu'=(6_a,5_a,5_b,4_b,4_a,3_a,2_a,1_b,1_a,1_b,1_b,0_c)\quad,\quad\nu''=(1_{c},4_{c},4_{c},4_{c},3_{c},1_{c})\]
and the insertion then gives 
\[\Omega^{-1}(\pi)=(6_a,5_a,5_b,4_{c},4_{c},4_{c},4_b,4_a,3_{c},3_a,2_a,1_{c},1_{c},1_b,1_a,1_b,1_b,0_c)\,\cdot\]
\end{ex}
\bi 
Let us now show that $\pi\in \Flt$.  First and foremost, note that 
\[\delta_g+|\mu_{s-1}| +s-1 = s\,,\]
and since $\nu'_u\leq s$ for all $u\in \{0,\ldots,s'-1\}$, we then obtain that $|\mu'_{s-1}| = |\mu_{s-1}|=1-\delta_g$. Besides,
 for all the $k\in \{0,\ldots,s-1\}$, we have that 
\begin{align*}
(\delta_g+|\mu_k| +k)-(\delta_g+|\mu_{k-1}| +k-1) &= 1+|\mu_k|-|\mu_{k-1}|\\
&=1-\ep(c_{k-1},c_k) \in\{0,1\}\,\cdot
\end{align*}
This means that the sequence $(\delta_g+|\mu_k| +k)_{k=0}^{s-1}$ is non-decreasing, and with the difference between consecutive terms at most equal to $1$, with equality if and only if $k\in D$.
\m On the other hand $u\in \{1,\ldots,s-1\}$, we have for all $u\in \{0,\ldots,s'-1\}$ that
\begin{align*}
\nu'_{u-1}-u-(\nu'_{u}-u-1) &= 1+\nu'_{u-1}-\nu'_u \geq 1\,\cdot
\end{align*}
The sequence $(\nu'_{u}-u-1)_{u=0}^{s'-1}$ is then decreasing.
\m Let us now set 
\[
 D_V = \{k\in \sss : |\mu'_{k-1}|-|\mu'_k| \neq |\mu_{k-1}|-|\mu_k|\}\,\cdot
\]
Since $\ep(c_{k-1},c_k)\in \{0,1\}$,
the set  $D_V$ then contains all the $k\in \{1,\cdots,s-1\}$ such that
\[
0<|\{0\leq u<s': \delta_g+|\mu_{k-1}| +k-1 \leq \nu'_{u}-u-1<\delta_g+|\mu_k| +k\}|\leq 1-\ep(c_{k-1},c_k)\,,
\]
so that $D_V \subset D$. For such $k$, there is a unique $u$ such that
\begin{equation}\label{eq:range}
\delta_g+|\mu_{k-1}| +k-1 \leq \nu'_{u}-u-1<\delta_g+|\mu_k| +k\,\cdot
\end{equation}
In fact, the sequence $(\nu'_{u}-u-1)_{u=0}^{s'-1}$ being decreasing, and the interval $[\delta_g+\mu_{k-1} +k-1,\delta_g+\mu_k +k)$, being a singleton for $k\in D_V$, contains at at most one element of the latter sequence. We also have that
\begin{align*}
|\{0\leq l<s: \delta_g+|\mu_l| +l \leq \nu'_{u}-u-1\}| &= |\{0,\ldots,k-1\}|=k\\
|\{0\leq v<s': \delta_g+|\mu_{k-1}| +k-1 \leq \nu'_{v}-v-1\}|&= |\{0,\ldots,u\}| = u+1\,\cdot
\end{align*}
Therefore, we have the following
\begin{align}
\nu'_{u} &= |\nu''_{u}|+k\,, \label{eq:recip}\\
|\mu'_{k}| &= |\mu_{k}|+u\,\nonumber\\
|\mu'_{k-1}| &=|\mu_{k-1}|+u+1\nonumber\,,
\end{align}
and by \eqref{eq:order} and \eqref{eq:range},
\begin{equation}\label{eq:insfor}
\mu'_{k-1}\gte\nu''_{u}\gte \mu'_{k}\,\cdot
\end{equation}
The part $\nu''_{u}$ is then inserted between $\mu'_{k-1}$ and $\mu'_{k}$. Note that this insertion occurs once for all $u$ such that 
$$|\overline{D}|=\delta_g+|\mu_0| \leq \nu'_{u}-u-1,$$
so that $$|D_V|=|\{0\leq k<s: \delta_g+|\mu_0|\leq \nu'_{u}-u-1\}|\,\cdot$$
\bi
Then, for all $u\geq |D_V|$, we have 
$$\nu'_{u}-u-1< \delta_g+|\mu_0|\,,$$
so that $\nu''_{u}=\nu'_u$. In particular, we have
\begin{equation}\label{eq:suppp}
 \nu''_{|D_V|}-|D_V|-1 < \delta_g+|\mu_0| \Longleftrightarrow  \nu''_{|D_V|}\leq \delta_g+|\mu_0| +|D_V|=\delta_g+|\mu'_0|\,\cdot
\end{equation}
We remark that
for all $k\in D\setminus D_{V}$, since $|\mu'_{k-1}|-|\mu'_{k}|= |\mu_{k-1}|-|\mu_{k}| =0$, the parts $\mu'_{k-1},\mu'_{k}$ then have the same size, and then the same relation with all parts with color $\co$. This means that, after inserting of the parts $\nu''_u$ into $\mu'$, we do not have any part between the parts $\mu'_{k-1}$ and $\mu'_{k}$. 
We also note that, for all $k\in \overline{D}\sqcup D_V$, $|\mu'_{k-1}|-|\mu'_{k}|=1$, so that one can insert any number of parts with color $\co$ and size $\delta_g+|\mu'_k|$, and since $\ep(\co,\co)=0$, these part with the same 
size and color $\co$ are well-related by $\gte$. 
\m These facts, along with \eqref{eq:insfor} and \eqref{eq:suppp}, imply that $\pi$ belongs to $\Flt$.
\bi We conclude by observing that, by \eqref{eq:insfor}, $D_V$ can be also defined as the unique subset of $D$ with satisfies the following: $k\in D$ belongs to $D_V$ if and only if there exists $u\in \{0,\cdots,s'\}$ such that $\mu'_{k-1}\gte\nu''_u\gte \mu'_k$.
\subsection{The maps are inverse each other}
Using \eqref{eq:retour1} and \eqref{eq:retour2}, we straightforward to observe by the definition of $\Omega^{-1}$ that $\Omega^{-1}\circ \Omega = Id_{\Flt(C)}$. On the other hand, the fact that $\Omega\circ \Omega^{-1} = Id_{\Reg(C)}$ comes from the correspondence between $D_W$ and $ D_V$. In fact, this correspondence is deduced from the equivalence between the definition of $W$ and the above definition of $D_V$. We also observe that the only parts whose size changes from one set of partitions to another are those related to the set $D_W$ and $D_V$. We finally conclude by observing the reciprocity between the definition of the weighted parts related to $D_W$ given in \eqref{eq:reci}, and the definition of the parts related to $D_V$  given by the formula \eqref{eq:recip}. 
\begin{rem}
The bijection built here gives a more refined property, as for a fixed color $C$ product of $s$ colors different from $\co$, it makes the correspondence between the partitions $\nu$ with at most $s$ parts such that the greatest part has size $s'$ and the flat partitions having $s'$ additional parts with colors $\co$ different from $0_{\co}$.
\end{rem}
\section{Bijective proof of \Thm{theo:flatreg2}}\label{sect:deg2}
In this section, we prove the following.
\begin{prop}\label{prop:same}
For a fixed color $C$ as product of colors different from $\co$ and a fixed non-negative integer $n$, the following sets of generalized partitions are equinumerous:
\begin{enumerate}
\item $\Fltt(C,n)=\{\pi\in \Fltt: C(\pi)_{|\co=1}=C,|\pi|=n\}$,
\item $\Flt(C,n)=\{\pi\in \Flt: C(\pi)_{|\co=1}=C,|\pi|=n\}$,
\item $\Reg(C,n)=\{\pi\in \Reg: C(\pi)_{|\co=1}=C ,|\pi|=n\}$,
\item $\Regg(C,n)=\{\pi\in \Regg: C(\pi)_{|\co=1}=C,|\pi|=n\}$.
\end{enumerate}
\end{prop}
In the previous section, we have shown in the proof of \Thm{theo:flatreg1} that $|\Flt(C,n)|=|\Reg(C,n)|$. In the following, we first show that there is a bijection between $\Fltt(C,n)$ and $\Flt(C,n)$, and after that we build the bijection between $\Reg(C,n)$ and $\Regg(C,n)$.
\subsection{Bijection between $\Fltt(C,n)$ and $\Flt(C,n)$}\label{sect:flat12}
Here we recall that, by \Def{def:flat2}, the partitions of $\Fltt$ have the form
$(\pi_0,\ldots,\pi_{s-1},0_{\co^2})$, such that for all $k\in \{0,\ldots,s-1\}$, we have $\pi_k\in \Sc$, and if we set $c(\pi_k)=c_{2k}c_{2k+1}\in \C^2$, we have by \eqref{eq:flatcond2} that
\begin{equation}\label{eq:flt2}
\mu(\pi_k)\gte\gamma(\pi_{k+1})\,\cdot
\end{equation}
We also observe that $c_{2s-2}c_{2s-1}\neq \co^2$, otherwise the latter equation gives  that
$|\pi_{s-1}|-|0_{\co^2}| = 4\ep(\co,\co)=0$, and then $\pi_{s-1} = 0_{\co^2}$, and this contradicts the definition of grounded partitions. Besides, we remark that $\mu(\pi_{s-1})=0_{\co}$ if and only if $c_{2s-1}=\co$.
\bi Let us consider the map $\mathcal F$ from $\Fltt$ to $\Flt$ defined by  
\begin{equation}
(\pi_0,\ldots,\pi_{s-1},0_{\co^2}) \mapsto \begin{cases}
(\gamma(\pi_0),\mu(\pi_0),\gamma(\pi_1),\mu(\pi_1),\ldots,\gamma(\pi_{s-2}),\mu(\pi_{s-2}),\gamma(\pi_{s-1}),0_{\co}) \ \ \ \ \ \ \ \ \ \ \ \ \ \ \text{if}\ \ c_{2s-1}=\co\\\\
(\gamma(\pi_0),\mu(\pi_0),\gamma(\pi_1),\mu(\pi_1),\ldots,\gamma(\pi_{s-2}),\mu(\pi_{s-2}),\gamma(\pi_{s-1}),\mu(\pi_{s-1}),0_{\co}) \ \ \text{if}\ \ c_{2s-1}\neq\co
\end{cases}\,\cdot
\end{equation}
It is easy to check that both the total energy and the sequence of primary colors are preserved. To show that $\mathcal F (\pi_0,\ldots,\pi_{s-1},0_{\co^2})\in \Flt$, we proceed according to whether $c_{2s-1}=\co$ or $c_{2s-1}\neq \co$. Note that by definition of the secondary parts, we have that for all $k\in \{0,\ldots,s-1\}$ that
\[
|\gamma(\pi_k)|-|\mu(\pi_k)|=\ep(c_{2k},c_{2k+1})\Longleftrightarrow \gamma(\pi_k)\gte \mu(\pi_k)\,\cdot
\]
\begin{itemize}
\item If  $c_{2s-1}=\co$, then the latter equation and \eqref{eq:flt2} give that $\mathcal F (\pi_0,\ldots,\pi_{s-1},0_{\co^2})$ is well-defined up to $\mu(\pi_{s-1})$,
 and with the fact that $c_{2s-2}\neq \co$ and  
$\mu(\pi_{s-1})=0_{\co}$, we obtain that $\mathcal F(\pi_0,\ldots,\pi_{s-1},0_{\co^2})\in \Flt$.
\item If  $c_{2s-1}\neq \co$, then the latter equation and \eqref{eq:flt2} give that $\mathcal F (\pi_0,\ldots,\pi_{s-1},0_{\co^2})$ is well-defined up to $\mu(\pi_{s-1})$, with the fact that $c_{2s-1}\neq \co$
and $\mu(\pi_{s-1})=\ep(c_{2s-1},\co)$, we obtain that $\mathcal F(\pi_0,\ldots,\pi_{s-1},0_{\co^2})\in \Flt$.
\end{itemize}
\bi The reciprocal map $\mathcal{F}^{-1}$ is even easier to build. We simply proceed as follows:
\begin{equation}
(\pi_0,\ldots,\pi_{s-1},0_{\co}) \mapsto \begin{cases}
(\pi_0+\pi_{1},\ldots,\pi_{s-1}+0_{\co},0_{\co^2}) \ \ \ \ \text{if}\ \ s \equiv 1 \mod 2\\\\
(\pi_0+\pi_{1},\ldots,\pi_{s-2}+\pi_{s-1},0_{\co^2}) \ \ \text{if}\ \ s \equiv 0 \mod 2
\end{cases}\,\cdot
\end{equation}
The primary parts being consecutive in terms of $\od$, the map $\mathcal{F}^{-1}$ is well-defined, and 
one can check that the first case of $\mathcal{F}^{-1}$ is reciprocal to the first case of $\mathcal{F}$, so as the second case of $\mathcal{F}^{-1}$ is reciprocal to the second case of $\mathcal{F}$. 
\subsection{Bijection between $\Reg(C,n)$ and $\Regg(C,n)$}\label{sect:oeground}
Let us recall that $\C'=\C\setminus \{\co\}$, and set $\Cc' = \{cc':c,c'\in \C'\}$. We now set $\rho = 1-\delta_g$ the common value of $\ep(c,\co)$ for all $c\in \C'$. Here we refer to $\Odd$ and $\E$ as the sets corresponding to the set $\C'$.
 We now show the following proposition. 
\begin{prop}\label{prop:regularity}
For a fixed color $C$ as product of colors in $\C'$ and a fixed non-negative integer $n$, the following sets of generalized partitions are equinumerous:
\begin{enumerate}
\item $\Reg(C,n)=\{\pi\in \Fltt: C(\pi)_{|\co=1}=C,|\pi|=n\}$,
\item $\Odd^{\rho_+}(C,n)=\{\pi\in \Odd^{\rho_+}: C(\pi)=C,|\pi|=n\}$,
\item $\E^{\rho_+}(C,n)=\{\pi\in \E^{\rho_+}: C(\pi)=C,|\pi|=n\}$,
\item $\Regg(C,n)=\{\pi\in \Regg: C(\pi)_{|\co=1}=C,|\pi|=n\}$.
\end{enumerate}
\end{prop}
By \Thm{theo:degree2}, we already have that $|\Odd^{\rho_+}(C,n)|=|\E^{\rho_+}(C,n)|$. We show in the remainder of this section that $\Reg(C,n)$ and $\Odd^{\rho_+}(C,n)$ are in bijection, as are $\E^{\rho_+}(C,n)$ and $\Regg(C,n)$.
\subsubsection{Bijection between $\Reg(C,n)$ and $\Odd^{\rho_+}(C,n)$}
This is straightforward by considering the following map from $\Reg(C,n)$ to $\Odd^{\rho_+}(C,n)$:
\begin{equation}
(\pi_0,\ldots,\pi_{s-1},0_{\co}) \mapsto (\pi_0,\ldots,\pi_{s-1})
\end{equation}
In fact, we have that $c(\pi_k)\in \C'$ for all $k\in \{0,\ldots,s-1\}$, and by \eqref{eq:diffcond}, that
$$|\pi_k|- |\pi_{k+1}|\geq \ep(c(\pi_k),c(\pi_{k+1}))\,,$$
 so that $|\pi_{s-1}|\geq \ep(c(\pi_{k+1}),\co)= 1-\delta_g = \rho$. By \eqref{rel} and \Def{def:rho}, we then have that the partition $(\pi_0,\ldots,\pi_{s-1})$ belongs to $\Odd^{\rho_+}(C,n)$. 
 \m The reciprocal map is obtain by adding a $0_{\co}$ to the right to a partition in $\Odd^{\rho_+}(C,n)$, and the latter relations imply that the resulting partition indeed belongs to $\Reg(C,n)$.
\subsubsection{Bijection between $\E^{\rho_+}(C,n)$ and $\Regg(C,n)$}
It may seem intricate to map these two sets, as a partition in the first set can have some primary part
while a partition in the second set cannot, but the regularity in $\co^2$ allows us to overcome this fact.
For simplicity, we write $\Sc(\C)$,  $\Sc(\C')$ and $\Pp(C')$ respectively the sets of the secondary parts with colors as a product of two primary colors in $\C$, the secondary parts with colors as a product of two primary color in $\C'$ and the primary parts with color in $\C'$.
We observe that we have a natural embedding $\Sc(\C')\hookrightarrow \Sc(\C)$.
\m For any $c\in \C'$, we have  by definition that the size of the secondary part with color $c\co$ has the same parity as $\ep(c,\co)=\rho$, while the size of the secondary part with color $\co c$ has the same parity as $\ep(\co,c)=1-\rho$. We then build the embedding $\Pp(C')\hookrightarrow \Sc(\C) $ as follows:
\[
k_c \mapsto \begin{cases}
k_{c\co} \ \ \text{if}\ \ k\equiv \rho \mod 2 \\
k_{\co c} \ \ \text{if}\ \ k\equiv 1-\rho \mod 2 
\end{cases}\,\cdot
\]
Therefore, we obtain a natural bijection $\Rr$ between $\Pp(C')\sqcup \Sc(\C')$ and $\Sc(\C)\setminus \{(2\Z)_{\co^2}\}$ with the relations

\begin{align}
\Sc(\C') \ni& (2k+\ep(c,c'))_{cc'}\mapsto (2k+\ep(c,c'))_{cc'}\\ \nonumber\\
\Pp(C') \ni& k_c \mapsto \begin{cases}
k_{c\co} \ \ \text{if}\ \ k\equiv \rho \mod 2 \\
k_{\co c} \ \ \text{if}\ \ k\equiv 1-\rho \mod 2 
\end{cases}\,\cdot\label{eq:rr}
\end{align}
We remark that the reciprocal bijection $\Rr^{-1}$ is also the identity on $\Sc(\C')$, and for a part with color $c\co$ or  $\co c$, we associate the part in $\Pp(\C')$ with the same size and color $c$.
\bi We can now extend the map $\Rr$ to the partitions in $\E^{\rho_+}$ with 
\begin{equation}
\Rr: (\pi_0,\ldots,\pi_{s-1})\mapsto (\Rr(\pi_0),\ldots,\Rr(\pi_{s-1}),0_{\co^2})\,,
\end{equation} 
and we have the following proposition.
\begin{prop}\label{prop:r}
The map $\Rr$ defines a bijection between $\E^{\rho_+}(C,n)$ and $\Regg(C,n)$. 
\end{prop}
Recall that $\odp$ in \Def{def:secpar} is the relation that relates the parts of a partition in $\E^{\rho_+}$, and the relation $\odg$ defined in \eqref{eq:diffcd2} relates the parts of a partition in $\Regg$. 
\m Note that the map $\Rr$ from $\Pp(C')\sqcup \Sc(\C')$ to $\Sc(\C)\setminus \{(2\Z)_{\co^2}\}$ conserves the size and the sequence of colors different from $\co$, so that extended to $\E^{\rho_+}$, it also preserves the total energy and the sequence of colors different from $\co$.
We now prove \Prp{prop:r} by using the two next lemmas.
\begin{lem}\label{lem:4}
Let us fix a color $c\in \C' \sqcup \Cc'$ and let us set $c=c(\pi_{s-1})$. We then have that the minimal size of $\pi_{s-1}\in \Pp^{\rho_+}\sqcup \Sc^{\rho_+}$ is the minimal size of $\Rr(\pi_{s-1})$ satisfying 
 $\Rr(\pi_{s-1})\gg 0_{\co^2}$.
\end{lem}
\begin{lem}\label{lem:5}
For all parts $k_p,l_q \in \Pp(C')\sqcup \Sc(\C')$, we have the following :
\begin{equation}
k_p\odp l_q \Longleftrightarrow \Rr(k_p)\odg \Rr(l_q)\,\cdot 
\end{equation}
\end{lem}
\Lem{lem:4} gives the equivalence of the minimal size condition for the last part, while 
\Lem{lem:5} states that the difference conditions are equivalent for both sets of partitions, and we directly obtain \Prp{prop:r}.
\begin{proof}[Proof of \Lem{lem:4}] We reason on whether $c\in \Cc'$ or $c\in\C'$  and $\pi_{s-1}$ with a size with the same parity as $\rho$ or $1-\rho$.
\begin{itemize}
\item If $c\in \Cc'$, we then write $c=c_0c_1$. We then have that
\begin{align*}
\pi_{s-1}\in \Sc^{\rho_+} &\Longleftrightarrow |\mu(\pi_{s-1})|\geq \rho & \text{by \Def{def:rho}}\\
&\Longleftrightarrow |\mu(\pi_{s-1})|\geq \ep(c_1,\co)\\
&\Longleftrightarrow \Rr(\pi_{s-1})=\pi_{s-1}\odg 0_{\co^2}\,\cdot& \eqref{eq:diffcd2}
\end{align*}
\item If $c\in \C'$ and $\pi_{s-1} \equiv \rho \mod 2$,
\begin{align*}
\pi_{s-1}\in \Pp^{\rho_+} &\Longleftrightarrow |\pi_{s-1}|\geq \rho \quad \text{and}\quad |\pi_{s-1}| \equiv \rho \mod 2 & \text{by \Def{def:rho}}\\
&\Longleftrightarrow |\pi_{s-1}|\in 2\Z_{\geq 0}+\rho\\
&\Longleftrightarrow  c(\mu(\Rr(\pi_{s-1})))=\co\quad \text{and}\quad  |\mu(\Rr(\pi_{s-1}))|\geq 0& \eqref{eq:rr}\\
&\Longleftrightarrow |\mu(\Rr(\pi_{s-1}))|\geq \ep(\co,\co)\\
&\Longleftrightarrow \Rr(\pi_{s-1})\odg 0_{\co^2}\,\cdot& \eqref{eq:diffcd2}
\end{align*}
\item If $c\in \C'$ and $\pi_{s-1} \equiv 1-\rho \mod 2$,
\begin{align*}
\pi_{s-1}\in \Pp^{\rho_+} &\Longleftrightarrow |\pi_{s-1}|\geq \rho \quad \text{and}\quad |\pi_{s-1}| \equiv 1+\rho \mod 2 & \text{by \Def{def:rho}}\\
&\Longleftrightarrow |\pi_{s-1}|\in 2\Z_{\geq 0}+1+\rho\\
&\Longleftrightarrow |\mu(\Rr(\pi_{s-1}))|\geq \rho \quad \text{and}\quad c(\mu(\Rr(\pi_{s-1})))=c& \eqref{eq:rr}\\
&\Longleftrightarrow |\mu(\Rr(\pi_{s-1}))|\geq \ep(c,\co)\\
&\Longleftrightarrow \Rr(\pi_{s-1})\odg 0_{\co^2}\,\cdot& \eqref{eq:diffcd2}
\end{align*}
\end{itemize}
One can observe that we always have the equivalence
\[\pi_{s-1}\in \Pp^{\rho_+}\sqcup \Sc^{\rho_+} \Longleftrightarrow \Rr(\pi_{s-1})\odg 0_{\co^2}\]
and this conclude the proof  of the lemma.
\end{proof}
\begin{proof}[Proof of \Lem{lem:5}]
Let us first state some obvious fact:
for all integer $a,b$, we have the following,
\begin{enumerate}
\item if $b\in \{-1,0,1\}$, then 
\begin{equation}\label{eq:chi1}
2a\geq b \Longleftrightarrow a \geq \chi(b=1)\,,
\end{equation}
\item if $b\in \{-2,-1,0\}$, then 
\begin{equation}\label{eq:chi-1}
2a\geq b \Longleftrightarrow a \geq -\chi(b=-2)\,\cdot
\end{equation}
\end{enumerate}
As before, we reason on the types of the parts $k_p$ and $l_q$.
\begin{itemize}
\item If $k_p\in \Sc$, we then write $k_p = (2u+\ep(c_0,c_1))_{c_0c_1}$.
\begin{itemize}
\item If $l_q\in \Sc$, we write $l_q=(2v+\ep(c_2,c_3))_{c_2c_3}$.
\begin{align*}
k_p\odp l_q &\Longleftrightarrow u-v- \ep(c_1,c_2)-\ep(c_2,c_3)\geq 0 &\eqref{ss}\\
&\Longleftrightarrow \Rr (k_p)\odg \Rr(l_q)\,\cdot&\eqref{eq:diffcd2}
\end{align*}
\item If $q\in \C'$ and $l\equiv \rho \mod 2$, we write $l_q=(2v+\ep(q,\co))_q$.
We then have 
\begin{align*}
k_p\odp l_q &\Longleftrightarrow  (2u+\ep(c_0,c_1))-(2v+\ep(q,\co)) \geq  1+\ep(c_0,c_1)+\ep(c_1,q) &\eqref{sp}\\
&\Longleftrightarrow  2(u-v-\ep(q,\co)-\ep(c_1,q))\geq \ep(\co,q)-\ep(c_1,q) \\
&\Longleftrightarrow  u-v-\ep(q,\co)-\ep(c_1,q)\geq \ep(\co,q)(1-\ep(c_1,q))&\eqref{eq:chi1} \\
&\Longleftrightarrow \Rr (k_p)\odg \Rr(l_q)\,\cdot&\eqref{eq:1}
\end{align*}
\item If $q\in \C'$ and $l\equiv 1-\rho \mod 2$, we write $l_q=(2v+\ep(\co,q))_q$.
\begin{align*}
k_p\odp l_q &\Longleftrightarrow  (2u+\ep(c_0,c_1))-(2v+\ep(\co,q)) \geq  1+\ep(c_0,c_1)+\ep(c_1,q) &\eqref{sp}\\
&\Longleftrightarrow  2(u-v-\ep(c_1,\co)-\ep(\co;q))\geq \ep(c_1,q)+\ep(\co,q)-1 \\
&\Longleftrightarrow  2(u-v-\ep(c_1,\co)-\ep(\co;q))\geq \ep(c_1,q)-\ep(q,\co) \\
&\Longleftrightarrow  u-v-\ep(c_1,\co)-\ep(\co;q)\geq \ep(c_1,q)\ep(\co,q) &\eqref{eq:chi1}\\
&\Longleftrightarrow \Rr (k_p)\odg \Rr(l_q)\,\cdot&\eqref{eq:1}\,\cdot
\end{align*}
\end{itemize}
\item If $p\in \C'$ and $k\equiv \rho \mod 2$, we write $k_p=(2u+\ep(p,\co))_p$
\begin{itemize}
\item If $l_q\in \Sc$, we write $l_q=(2v+\ep(c_2,c_3))_{c_2c_3}$.
We then have 
\begin{align*}
k_p\odp l_q &\Longleftrightarrow  (2u+\ep(p,\co))-(2v+\ep(c_2,c_3)) \geq  \ep(p,c_2)+\ep(c_2,c_3) &\eqref{ps}\\
&\Longleftrightarrow  2(u-v-\ep(\co,c_2)-\ep(c_2,c_3))\geq \ep(p,c_2)-\ep(p,\co)-2\ep(\co,c_2) \\
&\Longleftrightarrow  2(u-v-\ep(\co,c_2)-\ep(c_2,c_3))\geq (\ep(p,c_2)-1)-\ep(\co,p)\\
&\Longleftrightarrow  u-v-\ep(\co,c_2)-\ep(c_2,c_3))\geq -(1-\ep(p,c_2))\ep(\co,p)&\eqref{eq:chi-1}\\
&\Longleftrightarrow \Rr (k_p)\odg \Rr(l_q)\,\cdot&\eqref{eq:-1}
\end{align*}
\item If $q\in \C'$ and $l\equiv \rho \mod 2$, we write $l_q=(2v+\ep(q,\co))_q$.
We then have 
\begin{align*}
k_p\odp l_q &\Longleftrightarrow  (2u+\ep(p,\co))-(2v+\ep(q,\co)) \geq  1+\ep(p,q) &\eqref{pp}\\
&\Longleftrightarrow  2(u-v-\ep(\co,q)-\ep(q,\co))\geq \ep(p,q)-1 \\
&\Longleftrightarrow  u-v-\ep(\co,q)-\ep(q,\co)\geq 0 &\eqref{eq:chi1}\\
&\Longleftrightarrow \Rr (k_p)\odg \Rr(l_q)\,\cdot&\eqref{eq:diffcd2}
\end{align*}
\item If $q\in \C'$ and $l\equiv 1-\rho \mod 2$, we write $l_q=(2v+\ep(\co,q))_q$.
We then have 
\begin{align*}
k_p\odp l_q &\Longleftrightarrow  (2u+\ep(p,\co))-(2v+\ep(\co,q)) \geq  1+\ep(p,q) &\eqref{pp}\\
&\Longleftrightarrow  2(u-v-\ep(\co,\co)-\ep(\co,q))\geq \ep(p,q)+\ep(\co,p)-\ep(\co,q)\\
&\Longleftrightarrow  2(u-v-\ep(\co,\co)-\ep(\co,q))\geq \ep(p,q)\\
&\Longleftrightarrow  u-v-\ep(\co,\co)-\ep(\co,q)\geq \ep(p,q)&\eqref{eq:chi1}\\
&\Longleftrightarrow \Rr (k_p)\odg \Rr(l_q)\,\cdot&\eqref{eq:always}
\end{align*}
\end{itemize}
\item If $p\in \C'$ and $k\equiv 1-\rho \mod 2$, we write $k_p=(2u+\ep(\co,p))_p$.
\begin{itemize}
\item If $l_q\in \Sc$, we write $l_q=(2v+\ep(c_2,c_3))_{c_2c_3}$.
We then have 
\begin{align*}
k_p\odp l_q &\Longleftrightarrow  (2u+\ep(\co,p))-(2v+\ep(c_2,c_3)) \geq  \ep(p,c_2)+\ep(c_2,c_3) &\eqref{ps}\\
&\Longleftrightarrow  2(u-v-\ep(p,c_2)-\ep(c_2,c_3))\geq -\ep(p,c_2)-\ep(\co,p)\\
&\Longleftrightarrow  u-v-\ep(p,c_2)-\ep(c_2,c_3)\geq -\ep(p,c_2)\ep(\co,p)&\eqref{eq:chi-1}\\
&\Longleftrightarrow \Rr (k_p)\odg \Rr(l_q)\,\cdot&\eqref{eq:-1}
\end{align*}
\item If $q\in \C'$ and $l\equiv \rho \mod 2$, we write $l_q=(2v+\ep(q,\co))_q$.
We then have 
\begin{align*}
k_p\odp l_q &\Longleftrightarrow  (2u+\ep(\co,p))-(2v+\ep(q,\co)) \geq  1+\ep(p,q) &\eqref{pp}\\
&\Longleftrightarrow  2(u-v-\ep(p,q)-\ep(q,\co))\geq \ep(\co,q)-\ep(p,q)\\
&\Longleftrightarrow  u-v-\ep(p,q)-\ep(q,\co)\geq \ep(\co,q)(1-\ep(p,q))&\eqref{eq:chi1}\\
&\Longleftrightarrow \Rr (k_p)\odg \Rr(l_q)\,\cdot&\eqref{eq:1}
\end{align*}
\item If $q\in \C'$ and $l\equiv 1-\rho \mod 2$, we write $l_q=(2v+\ep(\co,q))_q$.
We then have 
\begin{align*}
k_p\odp l_q &\Longleftrightarrow  (2u+\ep(\co,p))-(2v+\ep(\co,q)) \geq  1+\ep(p,q) &\eqref{pp}\\
&\Longleftrightarrow  2(u-v-\ep(p,\co)-\ep(\co,q))\geq \ep(p,q)-1 \\
&\Longleftrightarrow  u-v-\ep(p,\co)-\ep(\co,q)\geq 0 &\eqref{eq:chi1}\\
&\Longleftrightarrow \Rr (k_p)\odg \Rr(l_q)\,\cdot&\eqref{eq:diffcd2}
\end{align*}
\end{itemize}
\end{itemize}
\end{proof}
\section{Degree beyond \Thm{theo:flatreg2}} \label{sect:degk}
We begin this section by defining a part of degree $k$.
\begin{deff}
Let $\C$ be a set of  primary colors. For any $k\in \Z_{\geq 1}$, we define the set of colors of degree $k$ as the set of the products of $k$ primary colors:
$$
\C^k = \{c_1\cdots c_k:c_1,\ldots,c_k\in \C\}\,\cdot
$$
For an energy $\ep$ and the corresponding flat relation $\gte$ defined on the set of primary parts, we define the set $\Pp^k = \Z \times \C^k$ of parts of degree $k$ as the sum of $k$ primary parts well-related by $\gte$:
\begin{equation}\label{eq:degkpar}
(p,c_1\cdots c_k) = \sum_{u=1}^k \left(p+\sum_{v=u}^{k-1}\ep(c_v,c_{v+1})\right)_{c_u} 
= \left(kp+\sum_{u=1}^{k-1}u\ep(c_u,c_{u+1})\right)_{c_1\cdots c_k}\,\cdot
\end{equation}
We set the function $\gamma_1,\ldots,\gamma_k$ on $\Pp^k$ such that
\begin{equation}\label{eq:gamk}
\gamma_i(p,c_1\cdots c_k) = \left(p +\sum_{u=i}^{k-1}\ep(c_i,c_{i+1})\right)_{c_i}\,,
\end{equation}
and we then obtain that 
\begin{align}
(p,c_1\cdots c_k) &= \sum_{i=1}^k \gamma_i(p,c_1\cdots c_k) \,,\label{eq:sumk}\\
\gamma_1(p,c_1\cdots c_k)&\gte\gamma_2(p,c_1\cdots c_k)\gte \cdots \gte \gamma_k(p,c_1\cdots c_k)\label{eq:ordk}\,\cdot
\end{align}
\end{deff}
\begin{deff}
We can then naturally define a flat relation $\gtrdot^k$ on $\Pp^k$ as follows:
\begin{align}\label{eq:flatcond3}
(p,c_1\cdots c_k)\gtrdot^k (q,d_1\cdots d_k)&\Longleftrightarrow p-q = \ep(c_k,d_1)+\sum_{u=1}^{k-1}\ep(d_u,d_{u+1})\,\nonumber \\
&\Longleftrightarrow \gamma_k(p,c_1\cdots c_k)\gte \gamma_1(q,d_1\cdots d_k)\,\cdot
\end{align}
The latter is equivalent to saying that the smallest primary part of $(p,c_1\cdots c_k)$ is greater than the greatest primary part of $(q,d_1\cdots d_k)$ in terms of $\gte$.
\end{deff} 
One can check that the relation $\gtrdot^k$ is indeed the flat relation linked to the energy $\ep^k$ defined on $\C^k \times \C^k $ by 
\begin{equation}\label{eq:epk}
\ep^k: (c_1\cdots c_k,d_1\cdots d_k) \mapsto \sum_{u=1}^{k-1}u\ep(c_u,c_{u+1}) + n\ep(c_k,d_1) + \sum_{u=1}^{k-1}(k-u)\ep(d_u,d_{u+1})\,\cdot
\end{equation}
In fact, by using \eqref{eq:degkpar} and \eqref{eq:flatcond3}, the difference of sizes of the parts  
$(p,c_1\cdots c_k)$ and $(q,d_1\cdots d_k)$ is exactly equal to $\ep^k(c_1\cdots c_k,d_1\cdots d_k)$.
\bi This extension of the flatness to degree $k$ has a strong connection with crystal base theory via the following proposition.
\begin{prop}
Let $\B$ be a crystal and suppose that there exists an energy function $H$ on $\B\ot\B$. Then, the function $H^k$  on  $\B^{\ot k}\ot \B^{\ot k} $ defined  by 
\begin{align}
b_1\ot \cdots \ot b_k \ot b_{k+1}\ot \cdots\ot b_{2k} &\mapsto \sum_{i=1}^{2k-1} \min\{i,2k-i\} H(b_i\ot b_{i+1})
\end{align}
is also an energy function  on  $\B^{\ot k}\ot \B^{\ot k} $.
\end{prop}
\begin{proof}
The tensor product being associative, we then have, for all $i\in \{0,\cdots,n\}$ and for all $j\in\{1,\ldots,2k\}$, that
\[\tilde e_i (b_1\ot\cdots\ot b_{2k} )= b_1\ot \cdots \ot \tilde e_i(b_j) \ot \cdots b_{2k}\Longrightarrow \begin{cases}
\tilde e_i (b_{j-1}\ot b_j)  = b_{j-1}\ot \tilde e_i(b_j) \\
\tilde e_i (b_j\ot b_{j+1})  = \tilde e_i(b_j)\ot b_{j+1}
\end{cases}\,\cdot\]
We thus obtain by \eqref{eq:ef} that, for $j\leq k$, (the following still holds for $j=1$)
\begin{align*}
H^k(\tilde e_i (b_1\ot\cdots\ot b_{2k} ))-H^k(b_1\ot\cdots\ot b_{2k})&= (j-1)\left(H(b_{j-1}\ot \tilde e_i(b_j))-H(b_{j-1}\ot b_j)\right)+\\
&\quad \quad j\left(H(\tilde e_i(b_j)\ot b_{j-1} )-H(b_j\ot b_{j+1})\right)\\
&= -(j-1)\chi(i=0)+j\chi(i=0)\\
&= \chi(i=0)\,\cdot
\end{align*}
On the other hand, for $j>k$, we have by \eqref{eq:ef} that (the following still holds for $j=2k$)
\begin{align*}
H^k(\tilde e_i (b_1\ot\cdots\ot b_{2k} ))-H^k(b_1\ot\cdots\ot b_{2k})&= (2k-j+1)\left(H(b_{j-1}\ot \tilde e_i(b_j))-H(b_{j-1}\ot b_j)\right)+\\
&\quad (2k-j)\left(H(\tilde e_i(b_j)\ot b_{j-1} )-H(b_j\ot b_{j+1})\right)\\
&= -(2k-j+1)\chi(i=0)+(2k-j)\chi(i=0)\\
&= -\chi(i=0)\,\cdot
\end{align*}
\end{proof}
The tensor of level $\ell$ perfect crystals being a level $\ell$ perfect crystal as well \cite{(KMN)$^2$a}, we then obtain that $\B^{\ot k}$ is a perfect crystal if $\B$ is.
\m
We note that the energy function of each perfect crystal $\B$ studied in \Sct{sect:level1} can be obtain by operating a transformation, which preserves the ground, on a certain minimal energy satisfying the condition in \Thm{theo:flatreg1} and such that $\delta_g=0$. Therefore, we can define both secondary flat and regular partitions corresponding to this energy function. In particular, since the corresponding minimal energy satisfies $\delta_g=0$, the energies related to these flat and regular partitions are almost equal by \eqref{eq:flatcond2} and \eqref{eq:diffcd2}. By \Prp{prop:r}, this means that the partitions, corresponding to those in $\E^{1_+}$ after applying the transformation on the minimal energy, satisfy some difference condition equal to the difference implied by the corresponding energy function of $\B^2$. In particular, one can view the case $A^{(2)}_{2n}$ as a result that links the generalization of the Siladi\'c theorem \cite{IK19} for $2n$ primary colors to the unique level one standard module $L(\Lambda_0)$. This fits with the original work of Siladi\'c \cite{Si02}, where he stated his identity after describing a basis of the unique level one standard module of $A_{2}^{(2)}$ through vertex operators. A suitable subsequent work is then to build the vertex operators, for the level one standard module of $A_{2n}^{(2)}(n\geq 2)$, which will allow us to describe a basis corresponding to the difference conditions given by the generalization of Siladi\'c's theorem.  
\bi 
We now define the degree $k$ flat partitions.
\begin{deff}
We define $\Fltk$ to be the set of \textit{degree} $k$ flat partitions, which are the flat partitions into degree $k$ parts in $\Pp^k$, with  ground $\co^k$ and energy $\ep^k$ defined in \eqref{eq:epk}.
\end{deff} 
In particular, when $\ep(\co,\co)=0$, we can then generalize the bijection built in \Sct{sect:flat12}.
\begin{prop}\label{prop:flat1k}
For any $k\geq 1$, there is a bijection $\mathcal{F}_k$ between $\Fltk$ and $\Flt$ that preserves the total energy and the sequence of colors different from $\co$ of the flat partitions.
\end{prop} 
\begin{proof}
For any flat partition $\pi=(\pi_0,\ldots,\pi_{s-1},0_{\co^k})$ in $\Fltk$, 
we associate the partition $\mathcal{F}_k(\pi)$ defined by the sequence
\[(\gamma_1(\pi_0),\ldots,\gamma_k(\pi_0),\gamma_1(\pi_1),\ldots,\gamma_k(\pi_1),\ldots,\gamma_1(\pi_{s-2}),\ldots,\gamma_k(\pi_{s-2}),\gamma_1(\pi_{s-1}),\ldots,\gamma_i(\pi_{s-1}),0_{\co})\,,\]
where $i=\max\{j\in \{1,\ldots,k\}:\gamma_j(\pi_{s-1})\neq 0_{\co}\}$. The existence of such index $i$ is ensured by the fact that $\pi_{s-1}\neq 0_{\co^k}$. It suffices to assume by contradiction that for all $j\in \{1,\ldots,k\}$ we have $\gamma_j(\pi_{s-1})=0_{\co}$. Since $\ep(\co,\co)=0$, we then have $0_{\co}\gte 0_{\co}$, and we obtain by \eqref{eq:sumk} that 
\[0_{\co^k}\neq \pi_{s-1}= \sum_{j=1}^k \gamma_j(\pi_{s-1})= \sum_{j=1}^k 0_{\co} = 0_{\co^k}\,\cdot\]
To prove that $\mathcal{F}_k(\pi)$ belongs to $\Flt$, we use \eqref{eq:ordk} along with \eqref{eq:flatcond3} to see that $\mathcal{F}_k(\pi)$ is well related up to $\gamma_i(\pi_{s-1})$, and to show that $\gamma_i(\pi_{s-1})\gte 0_{\co}$, we distinguish two cases.
\begin{itemize}
\item If $i<k$, we then have that $\gamma_{i+1}(\pi_{s-1})=0_{\co}$, and we conclude with \eqref{eq:ordk}.
\item If $i=k$, we then have by \eqref{eq:flatcond3} that $\gamma_{k}(\pi_{s-1})\gte \gamma_1(0_{\co^k})$ and we conclude.
\end{itemize}
\bi We now construct the inverse bijection $\mathcal{F}_k^{-1}$. For any $\pi =(\pi_0,\ldots,\pi_{s-1},0_{c_0})$, we write the decomposition $s=km-s'$ with the unique non-negative integers $m,s'$ such that $s'\in \{0,\ldots,k-1\}$. We then set 
\[\mathcal{F}_k^{-1}(\pi) = (\underbrace{\pi_0+\cdots+\pi_{k-1}},\underbrace{ \pi_k+\cdots+\pi_{2k-1}},\ldots, \underbrace{\pi_{(m-2)k}+\cdots+\pi_{mk-k-1}},\underbrace{\pi_{(m-1)k}+\cdots+\pi_{s-1}+ s' \times 0_{\co}}, 0_{\co^k})\,\cdot\]
Here we see by \eqref{eq:sumk}, \eqref{eq:ordk} and \eqref{eq:flatcond3}, this sequence is well-defined up to the part $\pi_{(m-1)k}+\cdots+\pi_{s-1}+ s' \times 0_{\co}$. Note that since $\pi_{s-1}\neq 0_{\co}$, we necessarily have that  $\pi_{(m-1)k}+\cdots+\pi_{s-1}+ s' \times 0_{\co}\neq 0_{\co^k}$. We distinguish two cases.
\begin{itemize}
\item If $s'>0$, since $\pi_{s-1}\gte 0_{\co}\gte 0_{\co}$, we then have by \eqref{eq:sumk} that $\pi_{(m-1)k}+\cdots+\pi_{s-1}+ s' \times 0_{\co}$ is  in $\Pp^k$, and by \eqref{eq:flatcond3} that 
$\pi_{(m-1)k}+\cdots+\pi_{s-1}+ s' \times 0_{\co}\gg^k 0_{\co^k}$.
\item If $s'=0$, we then have by \eqref{eq:sumk} that $\pi_{(m-1)k}+\cdots+\pi_{s-1}$ is  in $\Pp^k$, and since $\pi_{s-1}\gte 0_{\co}$, we obtain by \eqref{eq:flatcond3} that 
$\pi_{(m-1)k}+\cdots+\pi_{s-1}\gg^k 0_{\co^k}$.
\end{itemize}
The inversion comes from the correspondence between the case $s'=0$ for $\mathcal{F}_k^{-1}$ and $i=k$  for $\mathcal{F}_k$.
\end{proof}
\bi 
The latter proposition allows us to have the following correspondences
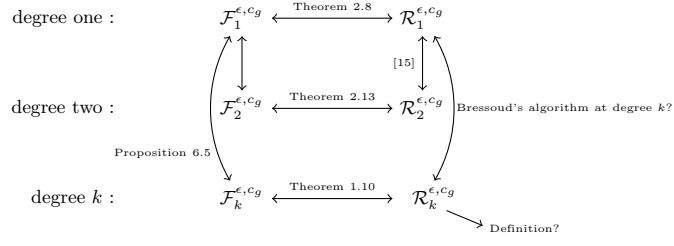
\begin{center}
\begin{tikzpicture}[scale=0.8, every node/.style={scale=0.8}]
\draw (0,0) node[left] {degree one :};
\draw (0,-1.5) node[left] {degree two :};
\draw (0,-3) node[left] {degree $k$ :};

\draw (2,0) node {$\Flt$};
\draw (2,-1.5) node {$\Fltt$};
\draw (2,-3) node {$\Fltk$};

\draw (5,0) node {$\Reg$};
\draw (5,-1.5) node {$\Regg$};
\draw (5.2,-3) node {$\Regk$};

\draw (6,-3.5) node[right] {{\tiny Definition?}}; 

\draw [->] (5.4,-3.2)--(6,-3.45);

\draw [<->] (2.5,0)--(4.5,0); \draw (3.5,0.2) node {{\tiny \Thm{theo:flatreg1}}};
\draw [<->] (2.5,-1.5)--(4.5,-1.5); \draw (3.5,-1.3) node {{\tiny \Thm{theo:flatreg2}}};
\draw [<->] (2.5,-3)--(4.5,-3); \draw (3.5,-2.8) node {{\tiny \Thm{theo:flatreg}}};

\draw [<->] (2,-0.3)--(2,-1.2);

\draw [<->] (1.8,-0.3) arc (150:210:2.4); \draw (1.6,-2.25) node[left] {{\tiny \Prp{prop:flat1k}}};

\draw [<->] (5,-0.3)--(5,-1.2); \draw (5,-0.75) node[left] {{\tiny \cite{IK201}}};

\draw [<->] (5.2,-0.3) arc (30:-30:2.4); \draw (5.45,-1.5) node[right] {{\tiny Bressoud's algorithm at degree $k$?}};

\end{tikzpicture}
\end{center}
A major subsequent work would be to find a suitable energy to define regular partitions for degree $k$ and which would allow us to state an analogue of \Thm{theo:flatreg} at degree $k$. This problem appears to be closely related to the problem of finding a generalization to weighted words at degree $k$ of the result stated in \cite{IK201}.
\section{Closing remarks}\label{sect:remarks}
We close this paper with two remarks.
First, we point out that in \cite{DK1,DK2}, Dousse and the author gave a theorem that connects some flat partitions to regular partitions. They considered the set of $n^2$ secondary colors $\C=\{\alpha_i\beta_j:i,j\in \{0,\cdots,n-1\}\}$, the ground $\co=a_0b_0$, and the energies $\ep,\ep_1,\ep_2$ defined by:
\begin{align}
\ep(a_ib_j,a_{i'}b_{j'})& = \chi(i\geq i')-\chi(i=j=i')+\chi(j\leq j')-\chi(j=i'=j')\\
\ep_1 (a_ib_j,a_{i'}b_{j'})& = \ep(a_ib_j,a_{i'}b_{j'})+\delta_1(a_ib_j,a_{i'}b_{j'})\\
\ep_2 (a_ib_j,a_{i'}b_{j'})& = \ep(a_ib_j,a_{i'}b_{j'})+\delta_2(a_ib_j,a_{i'}b_{j'})\,,
\end{align}
where $\delta_1$ equals $0$ up to the following exceptions,
\[
\delta_1(a_ib_j,a_{i'}b_{j'}) = 1 \quad\text{for}\quad 
\begin{cases}
i=j=i'=j'\neq 0\\
i=j=i'>i'\\
i<j=i'=j'
\end{cases}\,\cdot
\]
and $\delta_2$ equals $0$ up to the following exceptions,
\[
\delta_2(a_ib_j,a_{i'}b_{j'}) = 1 \quad\text{for}\quad 
\begin{cases}
i=j=i'=j'\neq 0\\
i=j=j'+1\leq i'\\
i\geq j+1=i'=j'
\end{cases}\,\cdot
\]
The flat partitions then correspond to those with energy $\ep$, and the first set of regular partitions corresponds to the energy $\ep_1$, and the second set to the energy $\ep_2$. However, the bijections established in \cite{DK1} link the flat partitions to some subsets of the sets of regular partitions:
\begin{itemize}
\item In the first case, the corresponding regular partitions are those which avoid the following forbidden patterns:
\begin{align*}
\text{for all}\quad i\geq  i'>j> j':& (p+1)_{a_ib_j},p_{a_{j'+1}b_{j'+1}},p_{a_{i'}b_{j'}}\,,\\
\text{for all}\quad i< i'<j\leq  j':& (p+1)_{a_ib_j},(p+1)_{a_{i+1}b_{i+1}},p_{a_{i'}b_{j'}}\,\cdot
\end{align*}
\item In the second case, the corresponding regular partitions are those which avoid the following forbidden patterns:
\begin{align*}
\text{for all}\quad i'> i>j'\leq j':& (p+1)_{a_ib_j},p_{a_{i'}b_{i'}},p_{a_{i'}b_{j'}}\,,\\
\text{for all}\quad i'\leq i<j'< j':& (p+1)_{a_ib_j},(p+1)_{a_{j}b_{j}},p_{a_{i'}b_{j'}}\,\cdot
\end{align*}
\end{itemize}
Furthermore, these bijections preserve the total energy of the partitions, but only the sequence of colors in which we replace $a_ib_i$ by $1$ for all $i\in\{0,\ldots,n-1\}$.
\m A suitable subsequent work could then be to investigate some analogue of \Thm{theo:flatreg} not only in terms of difference conditions implied by energies, but also in terms of forbidden patterns.
\bi 
The second remark concerns the notion of difference $d$ at distance $l$ for positive integers $d,l$.  A partition $\la=(\la_1,\cdots,\la_s)$ is said to satisfy the difference $d$  at distance $l$ condition if $\la_{i-k}-\la_k\geq d$ for all $i\in\{k+1,\ldots,s\}$. By considering the conjugated $\la^* = (\la^*_1,\ldots,\la^*_{r})$ of $\la$, we equivalently have that $\la^*_{i-d}-\la^*_i\leq k$ for all $i\in \{d,\ldots,r\}$. Therefore, a difference-distance condition imply a flatness condition. 
\m Using this term, the partitions with fewer than $m$ occurrences for each positive integer, described in the Glaisher theorem, are exactly the partitions satisfying difference $1$ at distance $m-1$. We can then see The Glaisher theorem as the link between partitions satisfying difference $1$ at distance $m-1$ to $m$-regular partitions. The Andrews-Gordon identities \cite{AN74}, a broad generalization of the Rogers-Ramanujan identities \cite{RR19}, can be seen as the level above Glaisher's identity, dealing with the partitions satisfying difference $2$ at distance $m-1$. They state that, for any positive integers $n,i,m$ such that $1\leq i\leq m$, the number of partitions of $n$ into parts not congruent to $0,\pm i \mod 2m+1$, is equal to the number of partitions of $n$ with fewer than $i$ occurrences of $1$, and which satisfy difference $2$ at distance $m-1$. Here again, the connection is made with a subset of $2m+1$-regular partitions. By \Thm{theo:kx}, this subset corresponds to the subset of all the $2m+1$-flat partitions with no parts congruent to $\pm i\mod 2m+1$. On the other hand, the subset of partitions satisfying the difference $2$ at distance $m-1$ corresponds to the subset of $m$-flat partitions $\la=(\la_1,\ldots,\la_s)$ satisfying the following: 
\begin{equation}
\begin{cases}
\la_1-\la_2 \leq i-1\,, \\
\la_{u-2}-\la_u \leq m-1 \ \ \text{for all} \ \ u\in\{3,s\}\,,\\
\la_s\leq m-1\,\cdot
\end{cases}
\end{equation} 
The problem of finding a simple bijective proof for the Andrews-Gordon identities, hence the Rogers-Ramanujan identities, could then be reduced to the problem of finding a bijection between the corresponding $m$-flat partitions  and $2m+1$-flat partitions, whose forms are relatively close one to each other. 
Other analogous identities given by Bressoud \cite{Bre80}, allow us to make a connection between subset of $m$-flat partitions and subsets of $2m$-flat partitions. Similarly, the identities conjectured by Kanade and Russell \cite{KR15}, proved for the cases related to $A_9^{(2)}$, partially by Bringmann, Jennings-Shaffer and Malhburg \cite{BJM19}, and completely by Rosengren \cite{Ro19}, can be interpreted in terms of identities  which link:
\begin{itemize}
\item for the case $A_7^{(2)}$, subsets of $3$-flat partitions to subsets of $9$-flat partitions,
\item for the case $A_9^{(2)}$, subsets of $4$-flat partitions to subsets of $12$-flat partitions.
\end{itemize}
An investigation in the framework of flat partitions could be a new way to approach the conjectured identities, and this could possibly lead to a more general family of such identities.

 \end{document}